\documentclass[12pt]{amsart}

\newtheorem{theorem}{Theorem}[section]
\newtheorem{lemma}[theorem]{Lemma}
\newtheorem{proposition}[theorem]{Proposition}
\newtheorem{corollary}[theorem]{Corollary}
\newtheorem{conjecture}[theorem]{Conjecture}
\theoremstyle{definition}
\newtheorem{definition}[theorem]{Definition}
\newtheorem{remark}[theorem]{Remark}
\newtheorem{question}[theorem]{Question}
\newtheorem{example}[theorem]{Example}

\newtheoremstyle{mine}
{8pt}
{8pt}
{}
{}
{\bfseries}
{}
{.5em}
{\thmname{#1}\thmnumber{ #2}\thmnote{ #3}.}
\theoremstyle{mine}
\newtheorem*{theorem*}{Theorem}

\usepackage{color,graphicx,amssymb,tikz,listings}
\usepackage{fullpage}
\usepackage{diagbox}

\usepackage[colorinlistoftodos]{todonotes}

\usepackage[utf8]{inputenc}
\usepackage[normalem]{ulem}

\newcommand{\quitado}[1]{{\color{red} }}

\usepackage{url}
\usepackage[colorlinks=true]{hyperref}
\usepackage{rotating}

\usepackage[nolabel]{showlabels}

\newcommand{\N}{\mathbb{N}}
\newcommand{\R}{\mathbb{R}}

\newcommand{\Z}{\mathbb{Z}}
\newcommand{\C}{\mathbb{C}}
\newcommand{\T}{\mathbb{T}}

\newcommand{\one}{\mathbf{1}}
\newcommand{\aff}{\operatorname{aff}}
\newcommand{\conv}{\operatorname{conv}}
\newcommand{\cone}{\operatorname{cone}}
\newcommand{\rank}{\operatorname{rank}}
\newcommand{\Int}{\operatorname{Int}}
\newcommand{\Star}{\operatorname{Star}}
\newcommand{\starop}{\operatorname{star}}
\newcommand{\UT}{\textup{UT}}
\newcommand{\SSB}{\textup{SSB}}
\DeclareMathOperator{\GL}{GL}
\DeclareMathOperator{\AGL}{AGL}

\renewcommand{\leq}{\leqslant}
\renewcommand{\geq}{\geqslant}
\renewcommand{\le}{\leqslant}
\renewcommand{\ge}{\geqslant}
\newenvironment{llave}
	{\[\left\{\begin{array}{rl}}
	{\end{array}\right.\]}
\newenvironment{eqllave}{\begin{equation}\left\{\begin{array}{rl}}
	{\end{array}\right.\end{equation}}
\newcommand{\mymatrix}[1]{\begin{pmatrix}#1\end{pmatrix}}

\title{Ewald's Conjecture and integer points in algebraic and symplectic toric geometry}
\author{Luis Crespo \,\,\,\,\,\, \'Alvaro Pelayo \,\,\,\,\,\, Francisco Santos}
\address{Luis Crespo, Francisco Santos,
	Departamento de Matem\'{a}ticas, Estad\'{i}stica y Computaci\'{o}n, Universidad de Cantabria, Av.~de Los Castros 48, 39005 Santander, Spain}
\email{luis.cresporuiz@unican.es, francisco.santos@unican.es}

\address{\'Alvaro Pelayo, 	Facultad de Ciencias Matem\'aticas,
	Universidad Complutense de Madrid, 28040 Madrid, Spain, and 
	Real Academia de Ciencias Exactas, F\'isicas y Naturales de Espa\~na
Calle Valverde, 22, 28004 Madrid, Spain}
\email{alvpel01@ucm.es}



\begin{document}

	\begin{abstract}
		We solve several open problems concerning  integer points of polytopes arising in symplectic and algebraic geometry.
		In this direction we give the first proof of a broad case of Ewald’s Conjecture (1988) concerning symmetric integral points of monotone lattice polytopes in arbitrary dimension. We also include an asymptotic quantitative study of the set of points appearing in Ewald's Conjecture. Then we relate this work to the problem of displaceability of orbits in symplectic toric geometry. We conclude  with a proof for 
		the $2$-dimensional case, and for a number of cases in higher dimensions, of Nill's Conjecture (2009), which is a generalization of Ewald's conjecture
		to smooth lattice polytopes. Along the way the paper introduces two new classes of polytopes which arise naturally in the study of
		Ewald's Conjecture and symplectic displaceability: neat polytopes, which are related to Oda's Conjecture, and deeply monotone polytopes.
			\end{abstract}
		
	\maketitle

\setcounter{tocdepth}{1}

\section{Introduction}

\emph{Smooth reflexive polytopes} are also known as \emph{monotone polytopes}, as they  are the images of monotone symplectic toric manifolds under of the  momentum map of symplectic geometry. 
Their number is finite in each dimension (modulo unimodular equivalence), but increases rapidly. See Table \ref{tablemon} in Section~\ref{ws} for the numbers up to dimension $9$, computed in~\cite{LorenzPaffenholz,OebroSmoothFano}, and Figure~\ref{fig-mon2} for a picture of the five monotone polytopes
 in dimension two.

\begin{figure}[h]
	\includegraphics[width=0.15\linewidth,trim=10cm 0 10cm 0,clip]{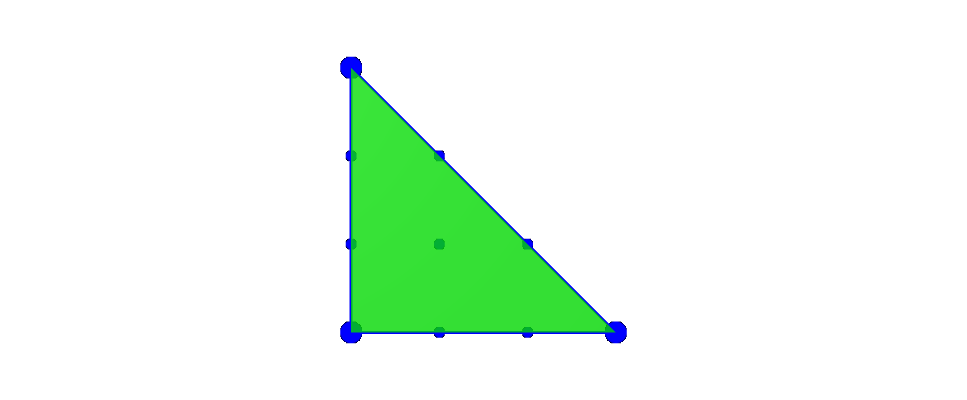}
	\includegraphics[width=0.15\linewidth,trim=10cm 0 10cm 0,clip]{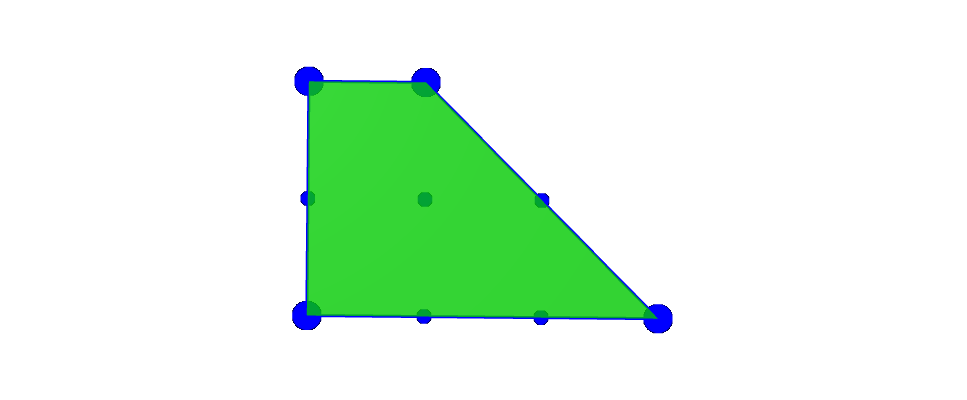}
	\includegraphics[width=0.15\linewidth,trim=10cm 0 10cm 0,clip]{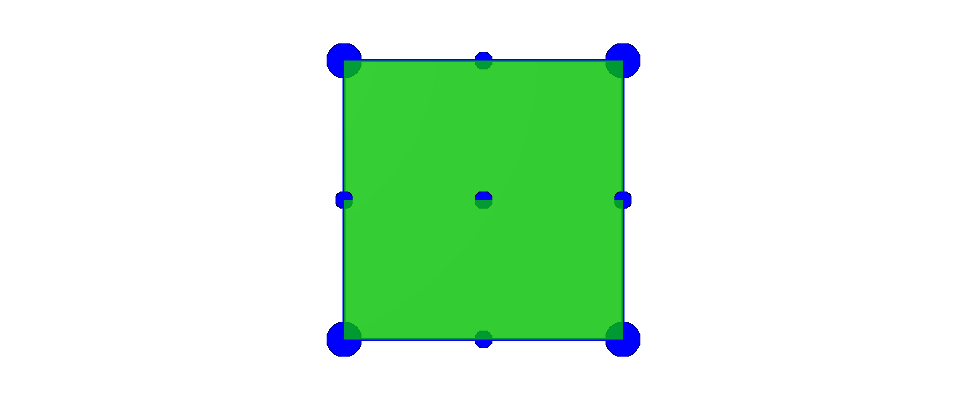}
	\includegraphics[width=0.15\linewidth,trim=10cm 0 10cm 0,clip]{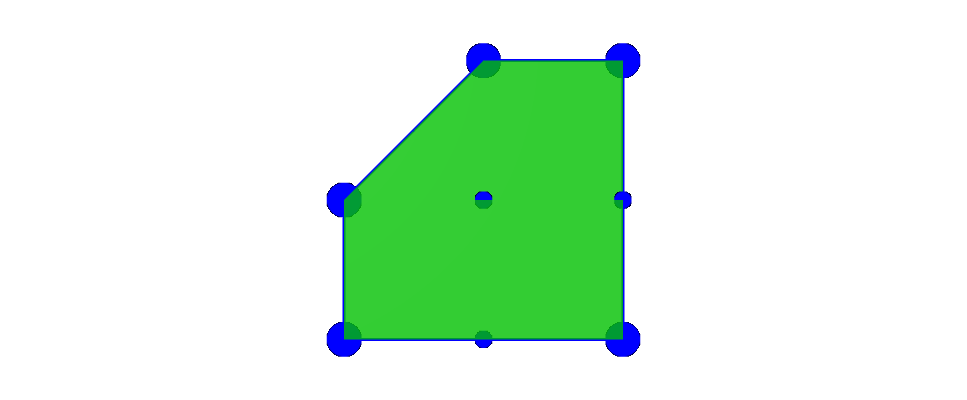}
	\includegraphics[width=0.15\linewidth,trim=10cm 0 10cm 0,clip]{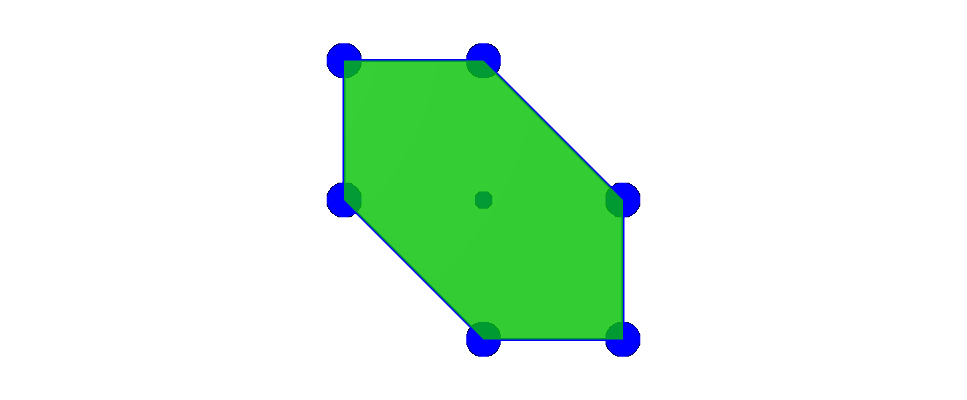}
	\caption{The five $2$-dimensional monotone polygons.
	We call them the monotone triangle, trapezoid, square, pentagon, and hexagon, respectively}
	\label{fig-mon2}
\end{figure}

Recall that a rational polytope is called \emph{smooth} if it is simple and every normal cone is unimodular, and \emph{reflexive} if it is a lattice polytope with the origin in its interior and its polar is also a lattice polytope.
Smooth polytopes in general, and monotone ones in particular, are very important in algebraic and symplectic geometry,
providing a strong link between ``discrete" problems in combinatorics/convex geometry and ``continuous" 
problems on smooth (toric) manifolds.

We refer to Charton-Sabatini-Sepe~\cite{CSS23}, Godinho-Heymann-Sabatini~\cite{GodinhoHeymannSabatini} and  McDuff~\cite{McDuff-topology},  
for recent works which discuss monotone polytopes, which are the momentum polytopes of \emph{monotone toric symplectic manifolds}, from the perspective of symplectic geometry and to
Batyrev~\cite{Batyrev}, Cox-Little-Schenck~\cite[Theorem 8.3.4]{CLS}, Franco-Seong~\cite{FrancoSeong}, Haase-Melnikov~\cite{HaaseMelnikov06} and Nill~\cite{Nill} for their relation to Gorenstein Fano varieties in algebraic geometry.

 In this paper we are interested in understanding, both theoretically and computationally, the properties of the  \emph{Ewald set} of a monotone polytope.
 This set appears implicitly in the influential 1988 paper by G\"unter Ewald~\cite{Ewald}. 

\begin{definition}[Ewald set]
The \emph{Ewald set} of a  polytope $P\subset \R^n$ is
\[
\mathcal{E}(P):=\Z^{n} \cap P\cap -P.
\]
Its points are called \emph{Ewald points} of $P$.
\end{definition}

That is, $\mathcal{E}(P) \subset \Z^n$ consists of the \emph{symmetric integral points of $P$}, 
meaning integral points $x \in \Z^n$ for which both $x\in P$ and $-x \in P$.
Our main motivation is the following  conjecture:%
\footnote{
The original formulation of Conjecture~\ref{ec} is in the dual, stating that the dual of any monotone polytope $P$ can be sent, via a unimodular transformation, to be contained in $[-1,1]^n$. 
	As pointed out by {\O}bro~\cite{Oebro-tesis} this is equivalent to our formulation, used already by McDuff~\cite[Section 3.1]{McDuff-probes} and Payne~\cite[Remark 4.6]{Payne}.
(McDuff and Payne remove the origin from $\mathcal{E}(P)$ in their definition, but for technical reasons we do not).   }

\begin{conjecture}[Ewald's Conjecture 1988 {\cite[Conjecture 2]{Ewald}}] \label{ec}
Let $n\in \N$. 
If $P$ is an $n$-dimensional monotone polytope in $\R^n$ then $\mathcal{E}(P)$ contains a unimodular basis of $\Z^n$.
\end{conjecture}
	
	The conjecture has been verified computationally for $n\leq 7$ by \O{}bro~\cite[page 67]{Oebro-tesis}, but little more is known about it. Both Payne and McDuff  \cite{McDuff-probes, Payne} remark that it is not even known whether there is a monotone polytope with $\mathcal{E}(P)=\{0\}$.

\begin{remark}
The Ewald set  of a rational polytope appears also in Payne's work on Frobenius splittings of toric varieties~\cite{Payne}. Although this is less related to our paper, let us state Payne's main result. Let $X$ be an $n$-dimensional (algebraic) toric variety, with associated fan $\Sigma_X$ in $\R^n$, and let $u_1,\dots,u_m$ be the primitive generators for the rays of $\Sigma_X$. (E.g., if $X$ is smooth then $\{u_i\}_i$ are the primitive facet normals for the corresponding smooth polytope $P$).
Let us call \emph{splitting polytope} of $X$ the (perhaps non-lattice) polytope
\[
P_X:=\{v \in \R^n \,|\, u_i\cdot v \le 1\ \forall i\}.
\]
Observe that all reflexive polytopes (in particular monotone ones) are examples of $P_X$ for some $X$. Payne's main result~\cite[Theorem 1.2]{Payne} says that $X$ is \emph{diagonally split} if and only if there is a prime $q\in \N$ such that $\mathcal{E}(\Int(q P_X))$ contains representatives for all the classes in $q \Z^n/\Z^n$.
\end{remark}

\subsection*{Three Ewald conditions and their relation to  symplectic  toric geometry}

\O{}bro's computational verification of Conjecture~\ref{ec} for $n\le 7$ shows the following strong version of it:  for every facet $F$ of $P$, $\mathcal{E}(P)\cap F$ contains a unimodular basis. We say that a polytope has the weak (resp.~strong) Ewald condition if Ewald's Conjecture \ref{ec} (resp. this strong form of it) holds for $P$. 

McDuff~\cite{McDuff-probes} introduces yet a third version of the Ewald property, that she calls \emph{star Ewald} (see Definition~\ref{ewaldcond}), motivated by the following problem in toric symplectic geometry.

It is known that every symplectic toric manifold $M$ has a particular \emph{central} toric orbit that is not \emph{displaceable} by a Lagrangian isotopy.
A relevant question is whether for a given manifold this central orbit is the only non-displaceable one. If this happens then the central orbit is called a \emph{stem}. 
 (See details on this in Section~\ref{sec:toric} and the references therein).
McDuff relates displaceability of toric orbits in $M$ to \emph{displaceability by probes}  of points in the corresponding momentum polytope (a concept that she defines). More precisely, she proves the following:

\begin{enumerate}
\item Let $M$ be a toric symplectic manifold with momentum polytope $P$.
If a point $u\in \Int(P)$ is displaceable by a probe then its fiber $L_u\subset M$ is displaceable by a Hamiltonian isotopy~\cite[Lemma 2.4]{McDuff-probes}.
\item A monotone polytope $P$ has the star Ewald property if and only if every point of $\Int(P)\setminus \{0\}$ is displaceable by a probe~\cite[Theorem 1.2]{McDuff-probes}.
\end{enumerate}

\begin{corollary}[McDuff~\cite{McDuff-probes}]
\label{cor:McDuff}
If the momentum polytope of a monotone sympletic toric manifold satisfies the star Ewald condition then the central fiber is a stem.
\end{corollary}

The star Ewald condition is stronger than the weak Ewald condition by \cite[Lemma 3.7]{McDuff-probes}. However, there are 6-dimensional monotone polytopes where the star Ewald condition fails~\cite[footnote to p. 134]{McDuff-probes} (see also our Proposition~\ref{paff}). Hence, the strong Ewald property 
does not imply the star Ewald property.

\subsection*{Nill's Conjecture}

Nill~\cite{cctv}  proposed to generalize Conjecture~\ref{ec} to smooth polytopes:

\begin{conjecture}[General Ewald's Conjecture, Nill 2009 {\cite{cctv}}] \label{gec}
	If $P$ is an $n$-dimensional smooth lattice polytope in $\R^n$ with the origin in its interior then ${\mathcal{E}}(P)$ contains a unimodular basis of $\Z^n$.
\end{conjecture}

This is, in principle,  stronger than Conjecture~\ref{ec}, but it might actually be equivalent; as Nill points out, Conjecture~\ref{ec} imples that ${\mathcal{E}}(P)$ linearly spans $\R^n$ for every smooth lattice polytope $P$ with $0\in \Int(P)$. (The implication is not in a dimension-by-dimension basis).

 \subsection*{Organization of this paper}
In Section \ref{sec:results} we discuss in some detail the structure of the paper and our results, and 
in Section~\ref{ws} we discuss preliminary results on monotone polytopes that  we need for the remaining of the paper. 
More specifically we recall basic facts about smooth, reflexive, and monotone polytopes, and we review the relation among the three Ewald properties.
After that, we have four sections containing our main results:

\begin{itemize}
\item
In Section~\ref{sec:deep} we show that the monotone polytopes containing the unit parallelepiped 
embedded at each vertex satisfy the three Ewald properties (Theorem~\ref{thm:deep-ewald}). We call them
\emph{deeply monotone} and, as a proof that they form a nice class, we show (Theorem~\ref{thm:deeply}) that they coincide with those that, recursively, do not contain unimodular triangles as faces and with those with the following property: for every face $f$ of $P$, the first displacement of $f$ (Definition~\ref{def:displacement}) has the same normal fan as $f$. 
\item
In Section~\ref{sec:bundles} we study the behavior of $\mathcal{E}(P)$ under taking \emph{polytope bundles}, which was already considered by McDuff. We extend some of her results, for example showing that a bundle is monotone if and only if both the fiber and base are monotone (McDuff proved one direction).
We also identify a certain property of a monotone polytope, that we call neatness, which is necessary for all bundles with fiber $P$ to satisfy Conjecture~\ref{ec} (and the star version of it) and sufficient for bundles with base a monotone segment. Thus, Conjecture~\ref{ec} would imply all monotone polytopes to be 
neat; interestingly, we show in Theorem~\ref{thm:oda} that this would also follow from a famous conjecture of Oda~\cite{HNPS2008,mfo2007,Oda1997}. Thus, it is natural to conjecture that all monotone polytopes are neat.
\item
In Section~\ref{sec:number} we study the number of Ewald points that a monotone polytope can have. It is easy to upper bound this by $3^n$, and we show there are monotone polytopes with only $3^{(2n+1)/3} \simeq 2.08^n$ of them.

\item
In Section~\ref{sec:Nill} we prove the more general Nill's Conjecture in dimension 2, and give partial results 
on it in dimension 3 and higher.
For example, we prove it for \emph{deeply smooth polytopes} if the origin is ``next to a vertex".
\end{itemize}

In
Section~\ref{sec:toric} we review the role of monotone polytopes in symplectic geometry and the implications of the Ewald conditions in this context. 
As an application, we state two corollaries of our results for symplectic toric manifolds in Theorems~\ref{thm:stem} and \ref{thm:symplectic}.

\subsection*{Acknowledgements}
We thank M\'onica Blanco (Universidad de Cantabria) for comments which improved the paper.
We thank Jo\'e Brendel for pointing out to us a connection of his paper \cite{Brendel} with our work, which has resulted in the inclusion of Theorem \ref{thm:symplectic} in the present paper.
We thank Andreas Paffenholz for communicating to us the example in the proof of Proposition~\ref{paff}, and Benjamin Nill for helpful conversations regarding Conjecture~\ref{gec}.

The first and third  authors are funded by grants PID2019-106188GB-I00 and PID2022-137283NB-C21 of
MCIN/AEI/10.13039/501100011033 / FEDER, UE and by project CLaPPo (21.SI03.64658) of Universidad
de Cantabria and Banco Santander. 

 The second author is funded by a BBVA (Bank Bilbao Vizcaya Argentaria) Foundation Grant for Scientific Research Projects with title \textit{From Integrability
to Randomness in Symplectic and Quantum Geometry}. He thanks the Dean of the
School of Mathematics Antonio Br\'u and the Chair
of the Department of Algebra, Geometry and Topology at the Complutense University of Madrid, Rutwig Campoamor, for their support
and excellent resources he is being provided with to carry out the BBVA project. He also thanks the Department of Mathematics, Statistics and Computation at the
University of Cantabria for inviting him in July and August 2023 for a visit during which
part of this paper was written, and the Universidad Internacional Men\'endez Pelayo (UIMP)
for the hospitality during his visit.

\quitado{
\emph{The second part of the paper is the ``algorithmic part"}, consisting only of one section, but which relies heavily on the previous sections. This part helps strengthening the connection between theory and 
computation in this context, via the construction of algorithms; more concretely:
\begin{itemize}
\item
In Section~\ref{algo}, which is divided into three parts, we give  algorithms which implement the
theoretical proofs of the main results of the paper.
\begin{itemize}
\item
Section~\ref{xc}  gives the algorithm we use to detect deeply monotone polytopes.
\item
Section~\ref{a1} gives the algorithm corresponding to the proof of Theorem~\ref{thm:deep-ewald}.
\item
Section~\ref{a2} gives the algorithm corresponding to the proof of Theorem~\ref{gec2}.
\end{itemize}
\end{itemize}
}

 \section{Main results}
 \label{sec:results}

\subsection*{Deeply monotone polytopes}

Section~\ref{sec:deep} is motivated by the following well-known property of smooth lattice polytopes, not shared by non-smooth ones: if a facet $F$ of such a polytope is moved one lattice unit towards the interior, the resulting polytope $F_0$ is still a lattice polytope. Moreover, if $F$ is reflexive then $F_0$ is automatically reflexive (Lemma~\ref{hyp-monotone}).
By ``moving a facet one unit'' we mean subtracting one to the right-hand side in its primitive facet-defining inequality. See details in Definition~\ref{def:displacement} and Lemma~\ref{hyp-monotone}, where $F'$ is called the \emph{first displacement} of $F$.

This would provide a recursive proof of Ewald's conjecture for all monotone polytopes, were it not for the fact that $F_0$ may in general no longer be smooth. We thus ask  what smooth lattice polytopes  are guaranteed to have all first displacements of facets and of lower-dimensional faces (which are needed for inductive arguments) smooth. Theorem~\ref{thm:deeply} says that this class can be characterized by any of the following equivalent properties:
\begin{itemize}
\item For every vertex $v$ of $P$, $P$ contains the unit parallelepiped generated by the edge-vectors from $v$.
\item For every face $f$ of $P$, the first displacement of $f$ has the same normal fan as $f$.
\item No unimodular triangle is a face of either $P$ or the first displacement of a face of $P$.
\end{itemize}

Based on the first characterization we call these polytopes \emph{deeply smooth} and, in case they are monotone, \emph{deeply monotone} (Definition~\ref{defi:deeply}). Our main result in Section~\ref{sec:deep} is:

\begin{theorem}[{Theorem \ref{thm:deep-ewald}}] 
	Every deeply monotone polytope satisfies the strong and star Ewald conditions (and, consequently, also the weak condition).
\end{theorem}

As far as we know this is the first proof of a broad case of Ewald's Conjecture in arbitrary dimension. In Table~\ref{ttt} we have computed how many monotone polytopes fall within this class for $n\le 6$, giving more than 1\,000 in dimension six.

\subsection*{(Monotone) fiber bundles and neat polytopes}
In Section~\ref{sec:bundles} we study fiber bundles. 
Given two  polytopes $B\subset  \R^k$ and $Q\subset \R^n$ a \emph{fiber bundle with base $B$ and fiber $Q$} is a polytope $P\subset  \R^{k+n}$ combinatorially isomorphic to $B\times Q$, that projects to $B$, and such that its intersection with $\{x\}\times \R^n$ has the normal fan of $Q$, for every $x\in B$. If $B$ is monotone we assume further that $Q$ is unimodularly equivalent to $\{0\}\times Q$.
See more details in Definition~\ref{defbundle}. 
\footnote{Our definition is equivalent to the ones in~\cite{McDuff-probes,McDTol} but our fibers and bases are swapped with respect to theirs due to the fact that their definition is phrased dually, in the normal fans.}

We offer an explicit description of bundles in ``canonical coordinates'' (Proposition~\ref{prop:bundle}) which allows us to show in Theorem~\ref{monbundle} that a bundle is smooth (resp. monotone) if and only if both its base and its fiber are smooth (resp. monotone). This extends \cite[Lemma 5.2]{McDuff-probes}, which shows only one direction.

We then look at the Ewald sets of bundles. If $P$ is a monotone bundle with fiber $Q$ and base $B$ then, with the natural identification $Q\cong \{0\}\times Q$, we have that  $\mathcal{E}(Q) \subset \mathcal{E}(P)$  (Proposition~\ref{prop:bundle-fiber}). It is then natural to ask under what conditions we have the analog property for the base: that every point in $\mathcal{E}(B)$ lifts to  $\mathcal{E}(P)$. The answer is the following:

\begin{definition}[Neat polytope]
\label{ineqs}
\label{def:pm}
	Let $m,n\in\N$. Let $P$ be a smooth lattice polytope in $\R^n$ defined by the inequalities $Ax\le c$, where $A\in\Z^{m\times n}$ and $c \in \Z^m$.
	For each $b\in \Z^m$ we define
	\[
	P_b:=\{x\in\R^n:Ax\le c+b\}
	\]
	and call it the displacement of $P$ by $b$. We say that $P$ is \textit{neat}
	if whenever $P_b$ and $P_{-b}$ are normally isomorphic to (i.e., have the same normal fan as)  $P$ for a $b\in\Z^m$ we have that
	\[
	P_b \cap (-P_{-b}) \cap \Z^n \ne \varnothing;
	\]
	that is,  there is an integer point $x\in P_b$ such that $-x\in P_{-b}$.
\end{definition}

Our main result in Section~\ref{sec:bundles} is that this condition is precisely what is needed in the fiber $Q$ for the Ewald properties to be preserved under fiber bundles:

\begin{theorem}[Corollary~\ref{cor:pmsym}]
	For a lattice smooth polytope $Q$ the following properties are equivalent:
\begin{enumerate}
\item $Q$ is  neat  and satisfies the weak (resp.~star) Ewald condition.
\item Every lattice smooth bundle $P$ with fiber $Q$ and base $[-1,1]$ satisfies the weak (resp.~star) Ewald condition.
\item Every lattice smooth bundle $P$ with fiber $Q$ and an arbitrary base $B$ satisfies the weak (resp.~star) Ewald condition whenever $B$ satisfies it.
\end{enumerate}	
\end{theorem}

\begin{corollary} \label{lastcor}
\label{second}
\begin{itemize}
	\item If Conjecture \ref{ec} holds then every monotone polytope is neat.
	\item If Conjecture \ref{gec} holds then every lattice smooth polytope is neat.
\end{itemize}
\end{corollary}

The  neat  property is also related to the following famous question of Oda, now considered a conjecture~(\cite{HNPS2008,mfo2007}), and open even in dimension three:

\begin{conjecture}[Oda, related to problems 1, 3, 4, 6 in \cite{Oda1997}]
\label{conj:Oda}
Let $P,Q\subset \R^n$ be lattice polytopes with $Q$ smooth and the normal fan of $Q$ refining that of $P$. Then,
\[
(P+Q) \cap \Z^n = P \cap \Z^n + Q \cap \Z^n,
\label{eq:mixedIDP}
\]
where $A+B :=\{a+b: a\in A, b\in B\}$ denotes the \emph{Minkowski sum} of two sets $A,B\subset \R^n$.
\end{conjecture}

\begin{theorem}
\label{thm:oda}
If Conjecture~\ref{conj:Oda} holds, then all smooth lattice polytopes containing the origin, in particular all monotone polytopes, are neat.
\end{theorem}

\begin{proof}
Let $P_b$ and $P_{-b}$ be displacements normally isomorphic to $P$. This implies that $P_b$ and $P_{-b}$ are still smooth, and smoothness implies them to be lattice polytopes: the system $A_v x = c_v\pm b_v$ defining a vertex $v$  has an integer solution for every integer $b$ since $\det(A_v) =\pm 1$.

We now show that $P_b + P_{-b}$ contains the origin, arguing by contradiction. If this was not the case, let $f$ be a linear functional that is strictly negative on $P_b + P_{-b}$ and assume it generic, so that it is maximized at a single vertex of $P_b + P_{-b}$. This vertex decomposes as $v+v'$ for the vertices $v\in P_b$ and $v'\in P_{-b}$ that maximize $f$ and we have that
\[
f(v) + f(v') = f(v+v') <0.
\]
This implies that one of $f(v)$ and  $f(v')$ is negative, contradicting that $0 \in P_b \cap P_{-b}$.

Now, since $P_b\cap\Z^n + P_{-b}$ contains the origin, Oda's conjecture implies that the origin decomposes as the sum of a lattice point in $P_b$ and another in $P_{-b}$. That is,  there is a lattice point $x\in P_b$ such that $-x\in P_{-b}$.
\end{proof}
That is, in view of our results Oda's Conjecture implies that all monotone polytopes are neat, and so, all monotone bundles of star Ewald polytopes are themselves star Ewald. Observe that in this proof we need only a weak version of Oda's Conjecture: the case where $P$ and $Q$ have the same normal fan.

Corollary~\ref{lastcor} and Theorem~\ref{thm:oda} lead naturally to the following conjecture:

\begin{conjecture}
\label{conj:pm}
All smooth lattice polytopes are neat.
\end{conjecture}

\subsection*{The number of Ewald points}

In Section~\ref{sec:number} we study the number of Ewald points that a monotone polytope can have. It is easy to show  (Proposition~\ref{prop:cube}) that for every monotone $n$-polytope 
\[
\mathcal{E}(P) \subset \mathcal{E}([-1,1]^n) =\{-1,0,1\}^n.
\]
Hence, no monotone $n$-polytope can have more than $3^n$ Ewald points. 

Somehow surprisingly, this number of Ewald points of the monotone cube is asymptotically attained (modulo a factor proportional to $\sqrt{n}$) also by the monotone simplex and by any bundle with fiber the monotone simplex and base a segment. 
We call the latter \SSB\ bundles,  short for ``simplex-segment bundles''.  
We describe \SSB\ bundles explicitly in Section~\ref{subsec:ssb} (see Proposition~\ref{prop-SSB}), where we show that they satisfy the three Ewald properties (Proposition~\ref{ssb-ewald}).
We  give a formula their exact number of Ewald points in Section~\ref{subsec:SBB-number} (Theorem~\ref{keyresult} and Proposition~\ref{prop:ssb-number}).

We have computed the number of Ewald points for every monotone polytope up to dimension seven. The full statistics (up to $n=5$) is in  Table~\ref{table:sizes}. The minimum number in each dimension is in Table~\ref{tablemin}.

\begin{table}[htb]
\begin{center}
	\begin{tabular}{|c|c|c|c|c|c|c|c|} \hline
		$n$ & 1 & 2 & 3 & 4 & 5 & 6 & 7 \\ \hline
		 $\min |\mathcal{E}(P)|$ & 3 & 7 & 13 & 27 & 59 & 117 & 243 \\ \hline
	\end{tabular}	
	\end{center}
\medskip
\caption{Minimum number of Ewald points among all monotone $n$-polytopes for $n\le 7$.
}
\label{tablemin}
\end{table}

This minimum number is particularly interesting for two reasons. One the one hand, it seems to grow exponentially  
 which  is evidence in favor of Conjecture~\ref{ec}. On the other hand, using fiber bundles, we can prove that for every $n$ there are monotone $n$-polytopes with approximately $3^{(2n+1)/3}\approx 2.08^{n}$ Ewald points:
 
 \begin{theorem}[Corollary \ref{cor:minimum}] 
 	For each  $n\in \Z_{\ge 1}$ let $\mathrm{E}_{\min}(n)$ denote the minimum number of Ewald points among monotone $n$-polytopes. Then:
\begin{equation*}
\label{eq:min}
\mathrm{E}_{\min}(n) \le 
\begin{cases}
3\cdot 9^k, & \text{if $n=3k+1$,}\\
\frac{59}9\cdot 9^k, & \text{if $n=3k+2$ and $k \ne 0$ (i.e., $n\ne 2$),}\\
13\cdot 9^k, & \text{if $n =3k+3$.}\\
\end{cases}
\end{equation*}
Hence,	
\[
	\lim_{n\to\infty} \sqrt[n]{\mathrm{E}_{\min}(n)} \le \sqrt[3]{9} = 2.08.
\]

 \end{theorem}

\subsection*{Nill's conjecture}
In Section~\ref{secgen} we study Nill's Conjecture~\ref{gec}. We prove a strong form of it in dimension 2, in which 
$P$ is allowed to be \emph{quasi-smooth} instead of smooth (Corollary~\ref{coro:dim2}). By quasi-smooth we mean that each vertex is at lattice distance one from the line spanned by its two neighboring boundary lattice points.

It seems quite challenging to make the type of arguments we use to work in dimensions $3$ or higher, but we provide two partial results:
We prove Conjecture~\ref{gec} for $n=3$ in the case where the origin lies in the first displacement of some edge (Proposition~\ref{prop:Nill-dim3}), and for all deeply smooth lattice polytopes in arbitrary dimension in the case where the origin equals the first displacement of some vertex (Proposition~\ref{prop:Nill-higherdim}).

\subsection*{Connection with symplectic/algebraic geometry}

Section~\ref{sec:toric} finishes the paper reviewing the relation between the Ewald conditions and symplectic geometry, and for those without
background in symplectic geometry we recommend reading that section before proceeding with this one.

 Via the Delzant correspondence $M\mapsto \mu(M)$ that sends a manifold $M$ to its momentum polytope $\mu(M)$,
 monotone polytopes correspond bijectively to the so called \emph{monotone symplectic toric manifolds}. McDuff's Corollary~\ref{cor:McDuff} implies that if $\mu(M)$ is star Ewald for a monotone manifold $M$ then the only non-displaceable toric orbit in $M$ is the central one. That is, the central orbit is a \emph{stem}.

With this in mind, the results in this paper imply for example that:

\begin{theorem}
\label{thm:stem}
Let $M$ be a monotone symplectic toric manifold with momentum polytope $P$. In the following cases the central orbit is a stem:
\begin{enumerate}
\item If $P$ is deeply monotone.
\item If $P$ is a bundle in which both the base and the fiber are star Ewald and the fiber is neat.
\end{enumerate}
\end{theorem}

\begin{proof}
Parts (i) and (ii) follow from Theorem~\ref{thm:deep-ewald} and Corollary~\ref{cor:pmsym}, taking Corollary~\ref{cor:McDuff}  into account.
\end{proof}

Finally, we would like to thank Jo\'e Brendel for pointing out to us his paper~\cite{Brendel} and the following consequences of our results.
Brendel says that a monotone polytope $P$ \emph{satisfies the FS property} if
$\mathcal{E}(P) \cap F \ne \emptyset$ for every facet $F$ of $P$. 

\begin{lemma}
\label{lemma:Brendel}
The weak Ewald property for a monotone polytope implies the FS property: $\mathcal{E}(P) \cap F \ne \emptyset$ for every facet $F$ of $P$.
\end{lemma}

\begin{proof}
Let $F$ be defined by the inequality $u\cdot x \le 1$ and suppose that $\mathcal{E}(P) \cap F = \emptyset$. Then,
$\mathcal{E}(P)\subset \{u\cdot x \le 0\}$, which implies $\mathcal{E}(P)\subset \{u\cdot x = 0\}$.
\end{proof}

We refer to the aforementioned paper by Brendel and the references therein for more details on the concepts (eg. Chekanov torus) which appear in the following result (see also Pelayo~\cite{PES} and Schlenk~\cite{SCH} for surveys on various aspects of symplectic geometry and topology).

\begin{theorem}\label{thm:symplectic}
	Let $(M,\omega)$ be a compact connected monotone symplectic toric manifold. Let us
	assume that $\mu(M)$ satisfies the weak Ewald property (e.g.,  it is deeply monotone). Then the following statements hold:
	\begin{itemize}
		\item[(1)]
		If the central fiber (that is, the fiber over the unique integral interior point of $\mu(M)$) is real, then $\mu(M)$ is centrally symmetric, that is, $\mu(M) =-\mu(M)$.
		\item[(2)]
		The Chekanov torus can be embedded into $M$ to yield an exotic Lagrangian which is not real.
	\end{itemize}
\end{theorem}

\begin{proof}
	Both results are proved by Brendel~\cite[Theorems 1.2 and 1.4]{Brendel} under the assumption that  $\mu(M)$ satisfies the FS property.
\end{proof}

In connection with this result, it is interesting to observe that centrally symmetric monotone polytopes are completely classified. It is proved in \cite{VosKly} that every centrally symmetric monotone polytope decomposes as a Cartesian product of \emph{del Pezzo polytopes} $DP_n$, where the $n$-dimensional del Pezzo polytope is the intersection of the  monotone $n$-simplex and its opposite. In other words, 
\[
DP_n:= \{(x_1,\dots,x_n)\in \R^n \,|\, |x_i|\le 1\, \forall i \text{ and } \left|\sum_i x_i\right| \le 1\}.
\] 
Del Pezzo polytopes are monotone if and only if $n=1$ or $n$ is even, and a Cartesian product is monotone if and only if its factors are monotone.

\section{Preliminaries on monotone polytopes} \label{ws}

Let $V$ and $V'$ by finite dimensional real vector spaces endowed with respective lattices $\Lambda$ and $\Lambda'$.
Given two polytopes $P\subset V$ and $P'\subset V'$ we say that $P$ and $P'$ are \emph{unimodularly equivalent} if there are  lattice subspaces $W\subset V$ and $W'\subset V'$ respectively containing $P$ and $P'$ and an affine isomorphism $f:W \stackrel{\cong}{\to} W'$ satisfying
\[
f(P) = P'
\quad{\text{and}}\quad
f(W\cap \Lambda) = W'\cap \Lambda'.
\]
Here a linear subspace $W$ is called a lattice subspace if it is linearly spanned by $W\cap \Lambda$.

All properties and results in this paper are invariant under unimodular equivalence. In practice, this implies that there is no loss of generality in assuming that $V=\R^n$,  $\Lambda=\Z^n$, and that $P$ is full-dimensional. We will typically make these assumptions in our definitions, statements and proofs. Observe that in this situation the map $f$ belongs to $\AGL(n,\Z)$.
That is, two polytopes $P,P'\subset \R^n$ are  equivalent
if and only if there is an $n$-dimensional integer matrix $A$ with determinant $\pm 1$ and a $t\in \Z^n$ 
such that $P$ is sent to $P'$ by the mapping $x\mapsto Ax+t$. For this reason unimodular equivalence is sometimes called $\AGL(n,\Z)$-equivalence.

A second form of equivalence that we use is that two polytopes $P,Q\subset V$ in the same ambient space are called \emph{normally isomorphic} if they have the same normal fan.

\subsection{Monotone, i.e. smooth reflexive, polytopes}

We here recall the notions of smooth, reflexive and monotone polytope. 
A comprehensive source on the topic is~\cite{HNP-book}.

\begin{definition}[Smooth polytope]
	An $n$-dimensional polytope in $\mathbb{R}^n$ is \emph{smooth} if it satisfies the following three properties:
	\begin{itemize}
		\item
		it is \emph{simple}: there are precisely $n$ edges meeting a each vertex;
		\item
		it is \emph{rational}: it has rational edge directions (equivalently, rational facet normals);
		\item
		the primitive edge-direction vectors at each vertex form a basis of the lattice  $\Z^n$. 
	\end{itemize}
	Equivalently, a smooth  polytope is a polytope whose normal fan is simplicial, rational, and unimodular.
\end{definition}

\begin{remark}
	In the algebraic geometry literature smooth polytopes are typically required to have integer vertices. We do not assume that here, and write ``lattice smooth polytope'' when integer vertices are assumed. In the symplectic geometry literature (perhaps-non-lattice) smooth polytopes are called  \emph{Delzant polytopes}.
\end{remark}

Let $P$ be an $n$-dimensional rational polytope and let $F$ be a facet of $P$. We write $u_F$ for the primitive exterior normal vector to $F$. With this notation in mind, there are
constants $b_F \in \R$ such that the irredundant inequality description of $P$ is
\[
P=\{ x \in \mathbb{R}^n \,|\,  u_F \cdot x   \leq b_F, \,\, \text{where $F$ is a facet of $P$} \}.
\]
For the following definition recall that a \emph{lattice polytope} is a polytope with vertices in the lattice
and that every polytope $P$ with the origin in its interior has a \emph{dual}, defined as:
\[
P^\vee:=\conv \left\{\frac{u_F}{b_F}\,|\, u_F \cdot x   \leq b_F \,\, \text{defines a facet of $P$} \right\}.
\]
(Observe that the origin being in the interior implies $b_F>0$ for every facet $F$).

\begin{definition}[Reflexive polytope]
	A \emph{reflexive polytope} is a lattice polytope with the origin in its interior and whose dual polytope is also a lattice polytope. Equivalently, a lattice polytope is
	reflexive if and only if every facet-defining inequality is of the form $u_F \cdot x \le 1$, where $u_F$ is the primitive exterior normal vector to the facet.
\end{definition}

If $P$ is a reflexive polytope then $P$ has the origin as its unique interior lattice point. Hence, for reflexive polytopes $\AGL(n,\Z)$-equivalence is the same as $\GL(n,\Z)$-equivalence.

Equivalently, a reflexive polytope is a lattice polytope containing the origin and whose facets are all at distance one from the origin, according to the following definition.

\begin{definition}\label{def:distance}
If $H=\{ x \in \R^n \,|\,  u \cdot x   = b\}$ (with $u\in \Z^n$ and primitive) is a hyperplane with rational direction, we call \emph{(lattice) distance} form a point $x_0$ to $H$ the number $|b-u\cdot x_0|$.

If $F$ is a polytope of codimension one (e.g., a facet of a full-dimensional polytope) we call distance to $F$ the distance to the hyperplane $\aff(F)$.
\end{definition}

\begin{definition}[Monotone polytope]\label{defmon}
	A polytope is  \emph{monotone} if it is smooth and reflexive. 
\end{definition}
	
	Modulo unimodular equivalence, the number of reflexive polytopes, and hence the number of monotone polytopes, is finite in each dimension~\cite{LagariasZiegler}.
	Table \ref{tablemon} shows the number of monotone polytopes up to dimension 9, and Figure \ref{fig-mon2} pictures the five that exist in dimension 2.  
	The enumeration for $n \leq 8$ is due to
	{\O}bro~\cite{OebroSmoothFano} and for $n=9$ is due to Lorenz and Paffenholz~\cite{LorenzPaffenholz}.

Monotone $n$-polytopes are known to have at most $3n$ facets~\cite{Casagrande} and conjectured to have at most $6^{n/2}$ vertices~\cite[Conj.~7.23]{HNP-book}. Both bounds are attained by  Cartesian products of monotone hexagons.

\begin{table}[htb]
\begin{center}
	\begin{tabular}{|r|c|c|c|c|c|c|c|c|c|} \hline
		dimension & 1 & 2 & 3 & 4 & 5 & 6 & 7 & 8 & 9 \\ \hline
		 monotone polytopes & 1 & 5 & 18 & 124 & 866 & 7622 & 72256 & 749892 & 8229721 \\ \hline
	\end{tabular}	
	\end{center}
\medskip
\caption{Number of  monotone polytopes in each dimension up to $9$.
}
\label{tablemon}
\end{table}

The simplest example of a monotone polytope is \textit{the} monotone simplex.

\begin{definition}[Smooth  simplex, monotone simplex]
\label{def:simplex}
	The \emph{smooth unimodular simplex} or \emph{standard simplex} is the polytope
	$\delta_n:=\Big \{x\in \R^n: x_i\geq 0 \,\, \forall i ,\ \sum_{i=1}^n x_i \leq 1\Big\}$.
	Every smooth lattice simplex is unimodularly equivalent to an integer dilation of it, that is, to the following \emph{smooth simplex of size $k$}:
	\[
	k\delta_n:=\Big \{x\in \R^n: x_i\geq 0 \,\, \forall i ,\ \sum_{i=1}^n x_i \leq k\Big\},
	\]
	
	The only one that is monotone is (a translation of) the smooth simplex of size $n+1$:
	\[
	\Delta_n:=\Big\{x\in \R^n : x_i\geq -1 \ \forall i ,\ \sum_{i=1}^n x_i \leq 1\Big\} = -\one + 
	(n+1) \delta_n,
	\]
	We call it  the \emph{monotone simplex}.
\end{definition}

\subsection{Three Ewald conditions and their relation to toric symplectic geometry}

In order to define precisely McDuff's star Ewald condition we introduce the following notation:
Let $P$ be any polytope 
and let $\mathcal F$ and $\mathcal R$ be the sets of facets and \emph{ridges} (that is, faces of codimension two) of $P$. 
For a face $f$ of $P$ we denote:
\[
\Star(f)=\bigcup_{f\subset F \in \mathcal F} F ;\quad
\starop(f)=\bigcup_{f\subset R \in \mathcal R} R ;\quad
\Star^*(f)=\Star(f)\setminus\starop(f).
\]

	For example, for any facet $F$  we have  $$\Star(F)=\Star^*(F)=F.$$

\renewcommand{\theenumi}{\roman{enumi}}
\begin{definition}[Ewald conditions, McDuff {\cite[Definition 3.5]{McDuff-probes}}]
\label{ewaldcond}
	Let $P$ be an $n$-dimensional  polytope  with the origin in its interior. We say that:
	\begin{enumerate}
		\item $P$ satisfies the \textit{weak Ewald condition} if $\mathcal{E}(P)$ contains a unimodular basis of $\Z^n$.
		\item $P$ satisfies the \textit{strong Ewald condition} if for each facet $F$ of $P$ the set $\mathcal{E}(P)\cap F$ contains a unimodular basis of $\Z^n$.
		\item A face $f$ of $P$ satisfies the \textit{star Ewald condition} or \emph{is star Ewald} if there exists $\lambda\in \mathcal{E}(P)$ such that $\lambda\in\Star^*(f)$ and $-\lambda\not\in \Star(f).$
		\item $P$ satisfies the \emph{star Ewald condition} if every face of $P$ satisfies it.
	\end{enumerate}
\end{definition}

The rest of this section sketches the proof of the following:

\begin{theorem}[McDuff\cite{McDuff-probes}]
\label{thm:po}
Either  of the strong Ewald  or the star Ewald conditions for a monotone polytope imply the weak Ewald conditon.
However, the strong Ewald condition does not imply the star Ewald condition.
\end{theorem}

That the strong condition implies the weak one is obvious. For the star Ewald condition, 
McDuff in~\cite[Lemma 3.7]{McDuff-probes} shows the following more precise statement:
if a monotone polytope $P$ has a vertex $v$ such that every face of $P$ containing $v$ is star Ewald, then $P$ satisfies the weak Ewald condition. 
McDuff also proves  that if $\mathcal{E}(P)\cap F$
contains a unimodular basis for a facet $F$, then every codimension $2$ face of $P$ contained in $F$ satisfies the star Ewald condition \cite[Lemma 3.8]{McDuff-probes}.

The relation between the strong Ewald and the star Ewald conditions is not completely clear, as  pointed out in  \cite[page 14]{McDuff-probes}, but the following result of \O{}bro and example of Paffenholz show that the star condition is not implied by the strong one:

\begin{theorem}[\O{}bro {\cite[page 67]{Oebro-tesis}}]\label{Oebro}
	Every monotone polytope of dimension  $7$ or less satisfies the strong Ewald condition.
\end{theorem}

\begin{remark}
McDuff~\cite[p.~134]{McDuff-probes} mentions that \O{}bro has verified the strong Ewald condition up to dimension $8$, but in \O{}bro's thesis~\cite{Oebro-tesis} this is stated only up to dimension $7$.
\end{remark}

\begin{proposition}[Paffenholz, for $n=6$]
 \label{paff}
For every $n\ge 6$, there are  $n$-dimensional monotone polytopes satisfying the strong Ewald condition but  not the star Ewald condition.
\end{proposition}

\begin{proof}
For $n=6$ this follows from Theorem \ref{Oebro} and the following example, which is one of several found by Andreas Paffenholz and whose existence is mentioned in \cite[p.~134]{McDuff-probes}. We thank Paffenholz for providing it to us.
Let $P$ be the polytope defined by 
\[
P:=\left\{x\in \R^6 \,|\,\begin{pmatrix}
			-I \\ A
		\end{pmatrix}x\le\one\right\},
		\qquad\text {where }
		A=\begin{pmatrix}
			-1 & 0 & 0 & 1 & 0 & 0 \\
			-1 & 0 & 1 & 2 & 0 & 0 \\
			-1 & 1 & 1 & 3 & 1 & 0 \\
			1 & 0 & 0 & -2 & 0 & 1
		\end{pmatrix}
		\]
		and $I$ is the $6$ by $6$ identity matrix.
		It can be checked that the vertex $F=\{(-4,1,1,-2,1,1)\}$ given by the intersection of the facets in the positions 2, 3, 5, 6, 9 and 10 is not star Ewald. 
		
		The example easily generalizes to higher dimensions. If $Q$ is any other monotone polytope then $P\times Q$ is also monotone, and if $Q$ satisfies the strong Ewald condition (e.g., $Q=[-1,1]^n$) then $P\times Q$ satisfies it as well, as we show in Proposition~\ref{prop:cartesian}. However, the face $F\times Q$ of $P\times Q$ cannot be star Ewald  because
	\[
	\Star(F\times Q)=\Star(F)\times Q,
	\qquad \text{and} \qquad
	\Star^*(F\times Q)=\Star^*(F)\times Q.
	\]
	In particular,  if there was an integer point $p\in\Star^*(F\times Q)\setminus(-\Star(F\times Q))$ then the projection of $p$ to $P$ would be an integer point in $\Star^*(F)\setminus(-\Star(F))$. 
	\qedhere
\end{proof}

\section{Deeply smooth polytopes, face displacements, and the Ewald conditions} 
\label{s2}
\label{sec:deep}

In this section we identify an interesting class of monotone polytopes and prove the three versions of Ewald's Conjecture for them.

\subsection{Face displacements in smooth polytopes}

The following definition is important not only in this section but also in subsequent ones:

\begin{definition}
\label{def:displacement}
Let $F$ be a face of codimension $k$ a smooth lattice polytope $P$, obtained as the intersection of $k$ facets with primitive facet inequalities $u_i\cdot x \le b_i$, for $i=1,\dots,k$ (with $u_i$ primitive). We call \emph{first displacement} of $F$ the polytope
\[
P \cap \{x \in \R^n \,|\, u_i\cdot x = b_i -1,\,\forall i=1,\dots,k\}.
\]
\end{definition}

Observe that if $P$ is monotone then $b_i=1$ for all $i$, so the first displacement of $F$ contains the origin. 
The following transitivity of first displacements is obvious:

\begin{lemma}
\label{lemma:transitive}
Let $F,G$ be faces of a smooth lattice polytope $P$, with $F\subset G$. Let $F_0$ and $G_0$ be their first displacements. Then, $F_0$ is also the first displacement of $F$ as a face of $G$.
\qed
\end{lemma}

\begin{lemma}
\label{hyp-monotone}
Let $P$ be an $n$-dimensional smooth lattice polytope and  let $F$ be a facet of $P$. Let $F_0$ be the  first displacement of $F$. Then:
	\begin{enumerate}
		\item $F_0$ is a lattice polytope.
		\item If $P$ is monotone (that is, reflexive) then $F_0$ is reflexive.
		\item $F_0$ is normally isomorphic to $F$, except perhaps if $P$ has a $2$-face that is a unimodular triangle with an edge in $F$ and third vertex in $F_0$.
	\end{enumerate}
\end{lemma}
\begin{proof}
	\begin{enumerate}
		\item 
		Let $H$ be the affine hyperplane containing $F_0$, which is at lattice distance one from $F$.
		For each edge $e$ incident in $F$, let $u_e$ be the endpoint of $e$ in $F$ and $v_e$ the intersection of $e$ with $H$.
		Since $P$ is simple, there is only one edge that leaves $F$ from each vertex, and since $P$ is smooth, the facets at $u_e$ form a unimodular basis, so $v_e$ has integer coordinates and lies in the edge. 
		Hence $F_0$ must have all the $v_e$ as vertices, and there are no more vertices not coming from $F$ (because every vertex of $F_0$ must have at least one edge towards the half-space containing $F$). In particular, all the vertices of $F_0$ are integer. 
		
		\item Assume without loss of generality that the facet $F$ is defined by the equation $x_n\leq 1$. Then, the inequality description of $F_0$ is the same as that of $P$, restricted to $x_n=0$; hence, $F_0$ is reflexive (some facet inequalities of $P$ may become redundant in $F_0$, but that does not affect the statement).

		\item If $v_e\ne v_f$ for every $e$ then $F_0$ is combinatorially isomorphic to $F$, hence simple. It is also smooth since the facet normals of $F_0$ at each vertex $v_e$ are the same as  those of $F$ at the corresponding vertex $u_e$.
		
		So, suppose that  $v_e=v_f$ for two edges $e,f$. Then, $u_e$ and $u_f$ must be connected with an edge, and the three points form a triangle, with the edge $u_e u_f$ in $F$ and its third vertex $v_e=v_f$ in $F_0$. The triangle has width one with respect to its edge in $F$, and the only smooth triangle of width one is the unimodular one (remember that $P$ is smooth and that every face of a smooth polytope is also smooth). \qedhere
	\end{enumerate}
\end{proof}

\begin{corollary}
If $P$ is monotone and contains no unimodular triangle as a $2$-face then $F_0$ is monotone for every facet $F$ of $P$.
\end{corollary}

\begin{example}\label{ex-inter}
	The third part of the Lemma, and the Corollary, may not hold if $P$ has a unimodular triangle as a face. For example, consider the following monotone polytope, for any $n\ge 3$ (it is the polytope $\SSB(n,n-1)$ introduced in Definition \ref{def-ssb}):
	\begin{llave}
		x_i & \ge -1 \quad\forall i,1\le i\le n; \\
		x_1 & \le 1; \\
		(n-1)x_1+x_2+\ldots+x_n & \le 1.
	\end{llave}%
Its intersection with the hyperplane $x_n=0$, parallel to the facet $x_n=-1$, gives a reflexive simplex:
\begin{llave}
	x_i & \ge -1 \quad\forall i,1\le i\le n-1; \\
	(n-1)x_1+x_2+\ldots+x_{n-1} & \le 1.
\end{llave}%
This simplex is not smooth: the normals to the facets in the vertex $(1,-1,\ldots,-1,0)$ are not a unimodular basis.

This example also shows that parts one and two of the lemma fail if $P$ is reflexive (but not smooth). By taking a second intersection, this time with $x_{n-1}=0$, we obtain another simplex
\begin{llave}
	x_i & \ge -1 \quad\forall i,1\le i\le n-2; \\
	(n-1)x_1+x_2+\ldots+x_{n-2} & \le 1,
\end{llave}
which has the non-lattice vertex $((n-2)/(n-1),-1,\ldots,-1,0,0)$.
\end{example}

Since absence of unimodular triangles is important, we give it a name.

\begin{definition}
\label{defi:UTfree}
A lattice polytope $P$ is called \UT-free if there is no $2$-face in $P$ that is a unimodular triangle.
\end{definition}

\begin{lemma}\label{lemma:UT-free}
Let $P$ be a \UT-free monotone polytope. Suppose that the first displacement of every facet of $P$
satisfies the weak (respectively strong) Ewald condition. Then $P$ satisfies the weak (respectively strong) Ewald condition too.
\end{lemma}

\begin{proof}
We first look at the weak Ewald condition.
Let $F_1$ and $F_2$ be two non-parallel facets of $P$ and let $P_1$ and $P_2$ be their first displacements.
	By Lemma \ref{hyp-monotone}, each $P_i$ is monotone and by hypothesis, $\mathcal{E}(P_i)$ contains a unimodular  basis $\mathcal{B}_i$ for the $(n-1)$-dimensional lattice spanned by it. Not all elements of $\mathcal{B}_1$ can lie in $P_2$, since $P_1\cap P_2$ is $(n-2)$-dimensional). Let $v \in \mathcal{B}_1 \setminus P_2$.
		Since $P$ is reflexive and $v,-v\in P\setminus P_2$,  $v$ is at distance $1$ from $P_2$, so $\mathcal{\mathcal{B}_2}\cup \{v\}\subset P$ is a unimodular basis for $\Z^{n+1}$.
		
		Now we consider the strong Ewald condition.
		 Let  $F$ any facet of $P$. Take two facets $F_1$ and $F_2$ adjacent to $F$ that are not parallel, and let $P_1$ and $P_2$ be their two interior displacemens and $F_i'=F\cap P_i$. 
		 By Lemma \ref{hyp-monotone}, $P_1$ and $P_2$ are monotone polytopes and $P_i$ is combinatorially isomorphic to $F_i$ for $i=1,2$. 
		 
		 As $F\cap F_i$ is a facet of $F_i$, $F_i'$ is a facet of $P_i$. By our hypotheses, $P_1$ and $P_2$ satisfy the strong condition, so $\mathcal{E}(F_i')$ contains a unimodular basis  $\mathcal{B}_i$ of the $(n-1)$-dimensional lattice spanned by $P_i$, for $i=1,2$.	 
		 (See Figure \ref{fig-bases} for an illustration of this process.) 
		 As in the first part, since $F_1$ and $F_2$ are not parallel
		 there must is a vector $v\in \mathcal{B}_1$ that is at distance $1$ from $P_1$ , so $\mathcal{B}_1\cup \{v\}$ is a unimodular basis contained in $F$.
\end{proof}

\begin{figure}[h]
\centerline{
	\includegraphics[width=0.33\linewidth,trim=3cm 0 3cm 0]{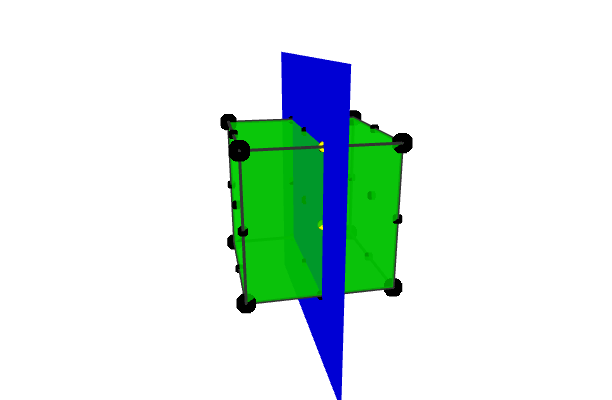}
	\includegraphics[width=0.33\linewidth,trim=3cm 0 3cm 0]{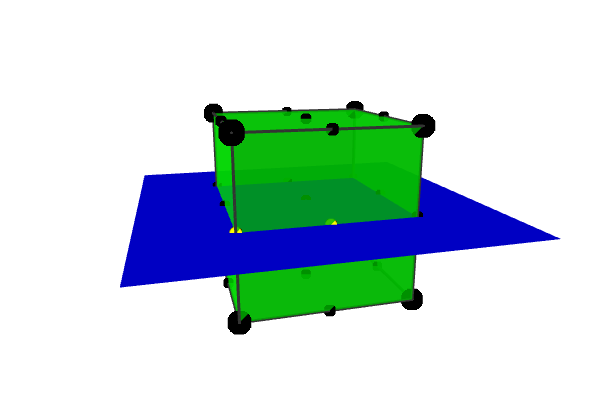}\qquad
	\includegraphics[width=0.3\linewidth]{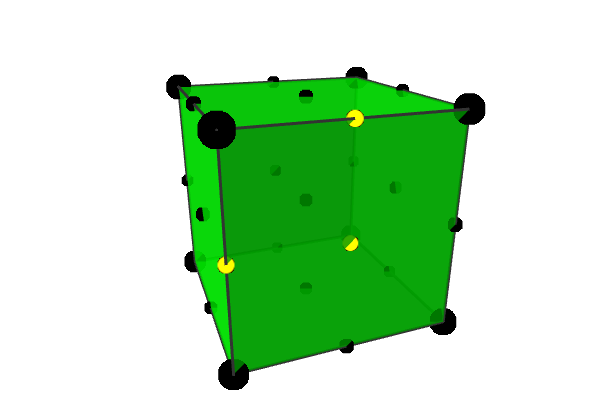}
}
	\caption{The first two figures show unimodular bases $\mathcal{B}_1$ and $\mathcal{B}_2$ (yellow points) of two hyperplanes (blue), as obtained in the proof of Lemma \ref{lemma:UT-free} for the case of the $3$-cube, where $F$ is the facet pointing forward. The third figure shows the resulting unimodular basis of $\Z^3$.}
	\label{fig-bases}
\end{figure}

\begin{corollary}
\label{cor:ewald-ind}
All $8$-dimensional \UT-free monotone polytopes satisfy the strong Ewald Condition.
More generally,
if all monotone polytopes in dimension $n$ satisfy the weak (respectively strong) Ewald condition then all $(n+1)$-dimensional monotone \UT-free polytopes satisfy the weak (respectively strong) Ewald condition.
\end{corollary}
\begin{proof}
This follows directly  from Lemma~\ref{lemma:UT-free} and Theorem \ref{Oebro}.
\end{proof}

\subsection{Deeply smooth polytopes satisfy the Ewald conditions}

\begin{definition}
\label{defi:deeply}
Let $v$ be a vertex of a lattice smooth polytope $P$, and let $u_1,\dots,u_n$ be the primitive edge vectors at $v$. We call \emph{corner parallelepiped} of $P$ at $v$ the parallelepiped
\[
\{v+\sum_{i=1}^n \lambda_i u_i\,|\, \lambda_i\in [0,1]\ \forall i\}.
\]

We say that $P$ is \emph{deeply smooth} if it contains the corner parallelepipeds at all its vertices.
We call $P$ \emph{deeply monotone} if it is deeply smooth and monotone.
\end{definition}

Observe that every lattice smooth polygon except the unimodular triangle is deeply smooth, and that all faces of a deeply smooth polytope are deeply smooth. In particular, deeply smooth polytopes are \UT-free.

\begin{theorem} \label{thm:deeply}
The following properties are equivalent for a smooth lattice polytope $P$:
\begin{enumerate}
\item $P$ is deeply smooth.
\item The first displacement of every face $F$ of $P$ is normally isomorphic to $F$.
\item $P$ and the first displacement of all its faces are \UT-free.
\end{enumerate}
\end{theorem}

\begin{proof}
The proof is by induction on the dimension, with the base cases of dimensions 1 and 2 being obvious because all polygons except the unimodular triangle satisfy the properties.

(i)$\Rightarrow$(ii): Since $P$ is deeply smooth, it is \UT-free. Hence, by Lemma~\ref{hyp-monotone}, the first displacement $F_0$ of every facet $F$ is a lattice polytope normally isomorphic to $F$. 

The corner parallelepipeds of $F_0$ are facets of the corner parallelepipeds of $P$, hence they are contained in $F_0$, so $F_0$ is also deeply smooth. Now, for an arbitrary face $G$ consider a facet $F$ containing it. By Lemma~\ref{lemma:transitive} the first displacement of $G$ in $P$ coincides with its first displacement in $F$, and by inductive hypothesis it is a lattice polytope normally isomorphic to $G$.

(ii)$\Rightarrow$(iii): Inductive hypothesis and Lemma~\ref{lemma:transitive} allows us to assume that the first displacement of every proper face of $P$ is \UT-free, so we only need to check that $P$ itself is $\UT$-free. Suppose that it is not, and let $F$ be a $3$-face of $P$ having a unimodular triangle as a face. Let $f$ be a facet of $F$ containing and edge $v_1v_2$ of that triangle but not the third vertex of it. Then the first displacement of $f$ in $F$ (and hence in $P$, again by Lemma~\ref{lemma:transitive}) is not normally equivalent to $f$ since it lacks the edge parallel to $v_1v_2$.

(iii)$\Rightarrow$(i): Suppose that $P$ is not smooth at a certain vertex $v$, and let $u_1,\dots,u_n$ be the primitive edge vectors at $v$. That is, there is an $I\subset\{1,\dots,n\}$ such that the point
\[
v_I:= v+\sum_{i\in I} \lambda_i u_i
\]
is not in $P$. Consider such an $I$ of smallest cardinality, and observe that this cardinality is at least two. Take $j,k\in I$ and  let $J=I\setminus\{j\}$, $K=I\setminus\{k\}$. Let $F$ be the face of $P$ spanned by $v$ and the vectors $u_i, i \in I\setminus \{j,k\} = J\cap K$. Then, the unimodular triangle with vertices $v_{J\cap K}$, $v_{J}$ and $v_{K}$ is a face in the first displacement of $F$. 
\end{proof}

\begin{example}
It is enlightening to check how each of the characterizations in Theorem~\ref{thm:deeply} imply the following:
The smooth lattice simplex $k\delta_n$, which is a smooth lattice polytope of $k\in \N$ (see Definition~\ref{def:simplex}), is deeply smooth if and only if $k\ge n$.
\end{example}

\begin{corollary}
\label{coro:deeply-faces}
If $P$ is deeply smooth then the first displacement of any face of it is deeply smooth (in particular, it is a smooth lattice polytope). 
\end{corollary}

\begin{proof}
Obvious, by Lemma~\ref{lemma:transitive} and characterizations (ii) or (iii) in Theorem~\ref{thm:deeply}.
\end{proof}

\begin{example}
All first displacements of (proper) faces of $P$ being deeply smooth lattice polytopes is not enough for $P$ to be deeply smooth. 
For an example, consider a smooth lattice tetrahedron of size three and let $P$ be obtained from it by cutting (``blowing up'') its four vertices at distance one. This gives a smooth lattice polytope with eight facets: four unimodular triangles (hence $P$ is not deeply smooth) and four monotone hexagons. 

The first displacements of all facets of $P$ are deeply smooth (they are smooth lattice triangles of size two); hence, by Corollary~\ref{coro:deeply-faces} and Lemma~\ref{lemma:transitive}, the first displacements of edges and vertices are deeply smooth too.
\end{example}

We can now prove the Ewald conditions for deeply smooth polytopes.

\begin{theorem}
\label{thm:deep-ewald}
Let $P$ be a deeply monotone polytope. Then 
\begin{enumerate}
\item If $u_1$ and $u_2$ are primitive edge vectors of $P$ at the same vertex $v$ then 
$\mathcal{E}(P)$ contains $u_1$, $u_2$, and at least one of $u_1+u_2$ and $u_1-u_2$.
\item $P$ satisfies the star Ewald condition.
\item $P$ satisfies the strong Ewald  condition.
\end{enumerate}
\end{theorem}

\begin{proof}
Let $f_1$ and $f_2$ be the edges at $v$ in the directions of $u_1$ and $u_2$.
Since $P$ is deeply smooth, the first displacement of $f_i$ is a monotone segment in the direction of $u_i$, that is, it is the segment $[-u_i,u_i]$. This shows that $u_1,u_2 \in  \mathcal{E}(P)$.
Now let $F_1$ be the facet containing $v$ but not $f_1$, and let $f_{12}$ be the 2-face containing $v$ and the edges with directions $u_1$ and $u_2$. Observe that $-u_1\in F_1$. 
The first displacement of $f_{12}$ is a monotone polygon containing $-u_1$ and intersecting $F_1$ in an edge of direction $u_2$. Since monotone polygons satisfy the strong Ewald condition, that edge contains 
(at least) another point of $\mathcal{E}(P)$. Hence, $\mathcal{E}(P)$ contains one of $-u_1+u_2$ and $-u_1-u_2$ (which is equivalent to the statement).

For part (ii), let $f$ be any face and $F$ a facet containing $f$. Let $v$ be a vertex of $f$ and $u$ the primitive vector for the edge at $v$ not contained in $F$. Then, $-u$ is a point in $F$ but not in $\starop(f)$, hence it is in $\Star^*(f)$. Since $u$ is not in $\Star(f)$ we are done.

For part (iii) let $F_1$ be a face and $v$ be a vertex of $F_1$. Let $u_1,\dots,u_n$ be the primitive edge vectors at $v$, with $u_1$ opposite to $F_1$. By part (i) the points $-u_1$ and $-u_1\pm u_i$, $i\in 2,\dots,n$ are in $\mathcal{E}(P)$. They clearly form an unimodular basis contained in $F_1$.
\end{proof}

\begin{remark}
In dimension $3$ all but two of the 18 monotone polytopes are deeply monotone (or \UT-free, which is equivalent for this dimension), but the proportion decreases rapidly with dimension and already in dimension $5$ less than half of the monotone polytopes 
are deeply monotone; see Table~\ref{table-utfree}. 
\end{remark}

\begin{figure}[h]
	\includegraphics[scale=0.3,trim=5.5cm 0.6cm 4.2cm 2cm,clip]{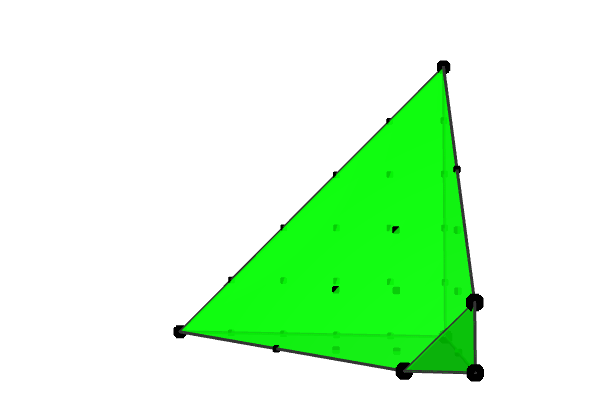}\qquad
	\includegraphics[scale=0.3,trim=5.5cm 0.4cm 4.2cm 2.5cm,clip]{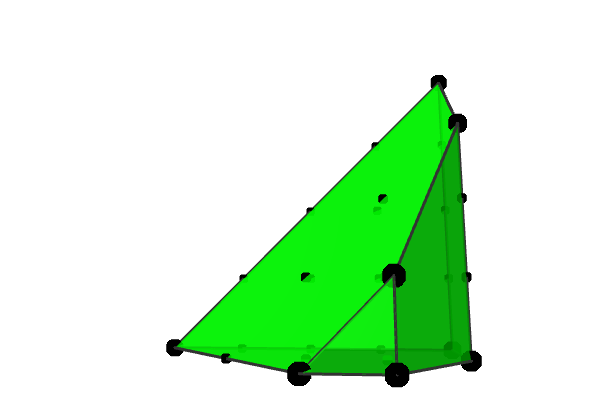}
	\caption{The only two non-\UT-free $3$-dimensional monotone polytopes. 
	}
	\label{noutfree}
\end{figure}

\begin{table}[h] \label{ttt}
	\begin{center}
		\begin{tabular}{|c|c|c|c|} \hline
			dimension & monotone & monotone \UT-free & deeply monotone \\ \hline
			3 & 18 & 16 & 16 \\
			4 & 124 & 74 & 72 \\
			5 & 866 & 336 & 300 \\
			6 & 7622 & 1699 & 1352 \\ \hline
		\end{tabular}
	\end{center}
	
	\medskip
	
	\caption{The number of polytopes of each class for each dimension. Theorem~\ref{thm:deep-ewald} says that Conjecture~\ref{ec} is true for 
		deeply monotone polytopes.}
	\label{table-utfree}
\end{table}

\section{Fiber bundles and neat polytopes} \label{bundlesection}
\label{sec:bundles}

Here we discuss the behavior of the fiber bundle operation in the context of neat polytopes. We also discuss interesting classes
of examples.

\subsection{Fiber bundles}

\begin{definition}[Bundle of a polytope]\label{defbundle}
	Let $n,k\in\N$. Given three polytopes $P\subset \R^{k+n}$, $B\subset \R^k$ and $Q\subset \R^n$, we say that $P$ is a \textit{bundle with base $B$ and fiber $Q$} if the following conditions hold:
	\begin{enumerate}
		\item $P$ is combinatorially equivalent to $B\times Q$.
		\item There is a short exact sequence of linear maps
		\[
		0\rightarrow \R^n\overset{i}{\longrightarrow} \R^{k+n}\overset{\pi}{\longrightarrow} \R^k\rightarrow 0
		\]
		such that
		$\pi(P)=B$
		and for every $x\in B$ we have that the polytope $Q_x:=\pi^{-1}(x)\cap P$ is normally isomorphic to 
		$i(Q)$.
			\end{enumerate}
\end{definition}

\begin{remark}
Our definition is equivalent to McDuff-Tolman~\cite[Definition 3.10]{McDTol} and \cite[Definition 5.1]{McDuff-probes} except there it is made in terms of the normal fans. The translation of our definition to their context is that the dual exact sequence 
\[
0\rightarrow (\R^k)^*\overset{\pi^*}{\longrightarrow} (\R^{k+n})^*\overset{i^*}{\longrightarrow} (\R^n)^*\rightarrow 0\]
injects the normal fan of $Q$ into that of $P$ and projects the normal fan of $P$ into that of $B$.
In particular, our use of ``base'' and ``fiber''  is reversed with respect to theirs.

Fiber bundles are also related to the following constructions:
\begin{itemize}
\item They are \emph{polytope bundles} in the sense of Billera-Sturmfels~\cite{BilStu}, defined in much the same way as fiber bundles except that $P$ does not need to be a polytope, or even convex, and the $Q_x$, although polytopes, may not be normally isomorphic to $Q$. 

\item They generalize the \emph{semidirect products} $B \ltimes_h Q$ of Haase et al.~\cite[Section 2.3.4]{HPPS2021}, which are the case in which all the fibers $Q_x$ are homothetic to $Q$; in their notation $h:B\to (0,\infty)$ is an affine function that gives the scaling factor at each $x\in B$.
\end{itemize}
\end{remark}

\begin{figure}[h]
	\includegraphics[width=0.32\linewidth,trim=3cm 0 3cm 0,clip]{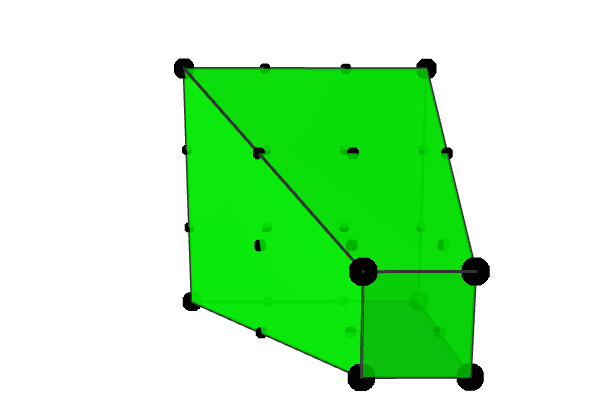}
	\includegraphics[width=0.32\linewidth,trim=3cm 0 3cm 0,clip]{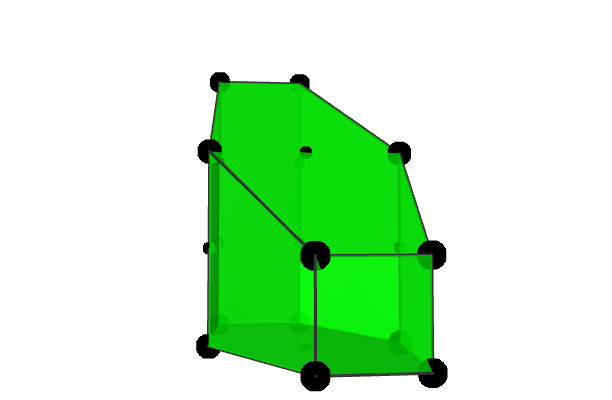}
	\caption{Two monotone bundles. The first has a segment as base and a square as fiber. The second has a hexagon as base and a segment as fiber.}
	\label{figbundles}
\end{figure}

\begin{example}
	A Cartesian product is a bundle, where any one of the two factors is the base and the other the fiber. See Figure \ref{figbundles} for two more examples. 
	All but three of the $18$ monotone polytopes decompose as bundles. The three that do not are the monotone tetrahedron and two ``wedges over a pentagon''.
\end{example}

The star Ewald case of the following result is \cite[Corollary 5.5]{McDuff-probes}:

\begin{proposition}
\label{prop:cartesian}
Let $P=P_1 \times \ldots\times P_s$ be a Cartesian product. Then:
\begin{enumerate}
\item $\mathcal{E}(P_1 \times \ldots\times P_s)=\mathcal{E}(P_1)\times \ldots \times \mathcal{E}(P_s)$.
\item $P$ satisfies the weak (resp. strong, star) Ewald condition if and only if
		every $P_i$ does.
\end{enumerate}
\end{proposition}

\begin{proof}
We prove each item separately.
	\begin{enumerate}
		\item We write the proof for two polytopes $P$ and $Q$. For any lattice point $(x,y)$, $(x,y)\in \mathcal{E}(P\times Q)$ means that $(x,y),-(x,y)\in P\times Q$, which in turn means that $x,-x\in P, y, -y\in Q$. On the other hand, $(x,y)\in \mathcal{E}(P)\times \mathcal{E}(Q)$ is the same as $x\in \mathcal{E}(P),y\in \mathcal{E}(Q)$, which is equivalent to $x,-x\in P, y, -y\in Q$.

		\item We start with the weak condition. If each $\mathcal{E}(P_i)$ contains a unimodular basis, their union gives a unimodular basis contained in $\mathcal{E}(P_1)\times \ldots \times \mathcal{E}(P_s)$.
		Conversely, if $\mathcal{E}(P_1)\times \ldots \times \mathcal{E}(P_s)$ contains a unimodular basis, its projection to each one of the $\mathcal{E}(P_i)$'s must contain a unimodular basis.
		
		For the strong condition, the reasoning is essentially the same, because a facet of $P_1 \times \ldots\times P_s$ is the product of a facet of a $P_i$ by all the other factors.
		
		For the star condition, the statement is proved in \cite[Corollary 5.5]{McDuff-probes}.
\qedhere
	\end{enumerate}
\end{proof}

\begin{proposition}[Canonical coordinates of a bundle]
\label{prop:bundle}
Let  $B\subset \R^k$ and $Q\subset \R^n$ be polytopes with facet descriptions
\[
B=\{x\in \R^k\,|\, u_i\cdot x \le b_i, i =1,\dots, l\},\qquad
Q=\{y\in \R^n\,|\, t_j\cdot y \le a_j, j =1,\dots, m\}.
\]
For each $j\in 1,\dots,m$, let $\phi_j:\R^k \to \R$ be an affine function and for each $x\in B$ consider the polytope
\[
Q_x:=\{ y\in \R^n\,|\, t_j\cdot y \le a_j + \phi_j(x), j =1,\dots, m\}.
\]
Finally, let $P\subset \R^{k+n}$ be defined by
\[
	P:=\left\{ (x,y)\in \R^k\times \R^n\,|\, 
	\begin{array}{ll}
	u_i\cdot x \le b_i,& i =1,\dots, l \\ 
	t_j\cdot y \le a_j + \phi_j(x), &j =1,\dots, m
	\end{array}
	\right\}.
\]
\begin{enumerate}
\item If $Q_x$ is normally isomorphic to $Q$ for every $x\in B$ then $P$ is a bundle with fiber $Q$ and base $B$.
\item All  bundles arise in this way, modulo a change of bases in $\R^k$, $\R^n$ and $\R^{n+k}$.
\end{enumerate}
\end{proposition}

\begin{proof}
For part (i), let $i:\R^n\to \R^{k+n}$ and $i:\R^{k+n}\to \R^{k}$ be the standard inclusion and projection. Then, 
\begin{align*}
\pi(P) &= \{x\in \R^k \,|\, \exists y\in \R^n\ \text{s.t.}\ x\in B, y \in Q_x\} = B, 
\end{align*}
and
$\pi^{-1}(x)\cap P = i(Q_x)$,
which is normally isomorphic to $i(Q)$ by hypothesis.
It only remains to show that $P$ is combinatorially equivalent to $B\times Q$. By construction, $P$ is defined by $l+m$ inequalities, as corresponds to a Cartesian product of polytopes with $l$ and $m$ inequalities respectively. We only need to show that the facet defined by each is combinatorially equivalent to $B$ times the corresponding facet of $Q$ or vice-versa, which we do by induction on $k+n$:
\begin{itemize}
\item For each $i\in \{1,\dots,l\}$, if $F_i$ is the facet of $B$ corresponding to the inequality $u_i\cdot x \le b_i$, induction on the dimension (replacing $\R^k$ with the hyperplane $H_i\cong \R^{k-1}$ containing $F_i$) gives us that the facet of $P$ with the same inequality is combinatorially isomorphic to $F_i\times Q$.

\item For each $j\in \{1,\dots,m\}$, if $G_j$ is the facet of $Q$ corresponding to the inequality $t_i\cdot y \le a_j + \phi_j(x)$, induction on the dimension (replacing $\R^n$ with the hyperplane $H_j\cong \R^{n-1}$ containing $G_j$) gives us that the facet of $P$ with the same inequality is combinatorially isomorphic to $B\times G_i$. In the induction step we are using now that since $Q_x$ is normally isomorphic to $i(Q)$, hence combinatorially equivalent to $Q$. 
\end{itemize}

For part (ii) suppose that $P$ is a  bundle and, by a change of bases, assume that $i$ and $\pi$ are the standard inclusion and projection. Let $x_0$ be a sufficiently generic point in $B$. Since $\pi^{-1}(x_0) \cap P$ is normally isomorphic to $i(Q)= \{0\}\times Q$, its normal vectors are $t_1,\dots,t_m$ (where we make the canonical identification $(\R^{k+n})^* =  (\R^k)^* + (\R^n)^*$).

Now, the fact that $x_0$ is generic and $P$ combinatorially isomorphic to $B\times Q$ with $\pi(Q)=B$ implies that the $j$-th facet of $\pi^{-1}(x_0) \cap P$ is contained in the facet of $P$ corresponding to $B\times F_j$, where $F_j$ is the $j$-th facet of $Q$. Hence, the facets of $P$ of that form have normal vectors of the form $t_j + s_j$, with $s_j \in (\R^k)^*$ and assumed primitive. That is, the facet  description of $P$ is
\[
	P:=\left\{ (x,y)\in \R^k\times \R^n\,|\, 
	\begin{array}{ll}
	u_i\cdot x \le b_i,& i =1,\dots, l \\ 
	t_j\cdot y  + s_j\cdot x \le a'_j , &j =1,\dots, m
	\end{array}
	\right\},
\]
for some $a'_1,\dots,a'_m\in \R$. This coincides with the statement, defining 
\[
\phi_j(x) = -s_j\cdot x +(a'_j-a_j).
\qedhere
\]
\end{proof}

\begin{remark}
\label{rem:zero-in-B}
The $Q$ in Definition~\ref{defbundle} is only important modulo normal equivalence. If $B$ contains the origin then we  assume, without loss of generality, that $i(Q) = \pi^{-1}(0)\cap P$. This fixes $Q$ (since $i:\R^k\to \pi^{-1}(0)$ is a linear isomorphism).
In this case the affine functions $\phi_j$ in the statement of Proposition~\ref{prop:bundle} become linear.
\end{remark}

We say that a linear map $f: \R^n\to \R^m$ is \emph{unimodular} if it respects the lattice in the following sense: If $\mathcal B$ is a unimodular basis of $\R^n$ that contains a unimodular basis of $\ker(f)$ then $f(\mathcal B)$ extends to a unimodular basis of $\R^m$, where $\{e_i\}_{i=1}^n$ denotes the standard basis in $\R^n$.
Put differently, modulo unimodular transformations in $\R^n$ and $\R^m$, $l$ is the linear map sending $e_i\mapsto e_i$ if $i\le \rank(f)$ and $e_i\mapsto 0$ otherwise.
(Observe that, in particular, $l(\Z^n)=l(\R^n)\cap \Z^m$).
For the rest of the paper all bundles are assumed unimodular, meaning that the linear maps $i$ and $\pi$ in the definition are unimodular.

Proposition \ref{prop:bundle} easily implies the following. The `only if' part is \cite[Lemma 5.2]{McDuff-probes}:

\begin{theorem}
\label{monbundle}
	Let $P$ be a bundle with base $B$ and fiber $Q$. Then:
	\begin{enumerate}
	\item $P$ is simple if and only if both $B$ and $Q$ are simple.
	\item $P$ is smooth if and only if both $B$ and $Q$ are smooth.
	\item $P$ is monotone if and only if both $B$ and $Q$ are monotone.%
	\footnote{Both sides of this statement imply $0\in B$, and in this statement we are implicitly assuming (as mentioned in Remark~\ref{rem:zero-in-B}), that $i(Q)=\pi^{-1}(0) \cap P$.
	}
	\end{enumerate}
\end{theorem}

\begin{proof}
Since being simple is a property of the combinatorial type, part (i) follows from the fact that a Cartesian product of polytopes is simple if and only if the factors are simple.

In part (ii), we assume without loss of generality that the product is in the canonical coordinates of Proposition \ref{prop:bundle}. In particular, the collection of normal vectors of $P$ is
\[
\{(u_1,0), \dots, (u_i,0), (t_1, s_1), \dots (s_j, t_j)\}.
\]
$P$ is smooth if, and only if, the subsets of normal vectors corresponding to facets meeting at each vertex of $P$ form matrices of determinant $\pm 1$. Observe that such matrices have shape
\[
\begin{pmatrix}
U&0\\
S&T\\
\end{pmatrix},
\]
where $U$ and $T$ are square matrices (corresponding to the facets of $B$ and $Q$ that meet at the corresponding vertices), so the determinant is the product of those of $U$ and $T$. Hence, the matrix is unimodular if and only if $U$ and $T$ are unimodular. That is, all normal cones at vertices of $P$ are unimodular if and only if all normal cones at vertices of $B$ and $Q$ are unimodular. This proves (ii).

For part (iii), the assumption $i(Q)=\pi^{-1}(x) \cap P$ makes the functions $\phi_j$ of Proposition \ref{prop:bundle} linear so, in the notation of the proof, $a'_j=a_j$. Hence, the inequality description of $P$ is
\[
	P:=\left\{ (x,y)\in \R^k\times \R^n\,|\, 
	\begin{array}{ll}
	u_i\cdot x \le b_i,& i =1,\dots, l \\ 
	t_j\cdot y  + s_j\cdot x \le a_j , &j =1,\dots, m
	\end{array}
	\right\}.
\]
Since, by part (ii), we can assume that $P$, $B$ and $Q$ are smooth, $P$ being monotone is equivalent to all $a_i$'s and $b_j$'s being equal to $1$, which in turn is equivalent to $B$ and $Q$ being monotone.
\end{proof}

We can say the following for the Ewald conditions of fiber bundles.

\begin{proposition}
\label{prop:bundle-fiber}
Let $P$ be a monotone bundle with fiber $Q$ and base $B$. Then
		$i(\mathcal{E}(Q)) \subset \mathcal{E}(P).$
		In particular,  
		\[
		\mathcal{E}(Q)\ne\{0\} \Rightarrow \mathcal{E}(P)\ne\{0\}.
		\]
\end{proposition}

\begin{proof}
		If $x\in i(\mathcal{E}(Q))$ then $x\in i(Q)\subset P$ and $-x\in i(Q)\subset P$. Hence $x\in \mathcal{E}(P)$.
\end{proof}

\subsection{Neat polytopes and Ewald's Conjecture} \label{s3b}

As pointed out by McDuff~\cite[Remark 5.7 and text above Proposition 1.4]{McDuff-probes} it seems likely that the total space of a monotone fiber bundle is star Ewald when the base and the fiber are star Ewald. She
proves the case when the fiber (which is called base there) is a simplex  \cite[Proposition 1.4 and Corollary 5.6]{McDuff-probes}. 

Here, we show that this holds for every base and fiber as long as the fiber is neat.
As mentioned in Section~\ref{sec:results} it is quite plausible that \emph{all} monotone polytopes have this property (Conjecture~\ref{conj:pm}) since a counter-example to this would violate both Ewald's Conjecture~\ref{ec} and Oda's Conjecture~\ref{conj:Oda}.
Hence, in view of Theorem~\ref{Oebro}, if such an example
exists, it has at least dimension $8$.

The following two results establish the  relations between the concept of neat polytope and the bundle operation on polytopes.
They are somehow similar to Proposition 5.3 and Lemma 5.4 in \cite{McDuff-probes}, except we state them for arbitrary smooth lattice polytopes and with respect to both the star Ewald and weak Ewald condition, while McDuff considers only monotone polytopes and the srat Ewald condition.

\begin{theorem}\label{thm:pmsym}
	For a lattice smooth polytope $Q$ the following properties are equivalent:
\begin{enumerate}
\item For every lattice smooth bundle $P$ with fiber $Q$ and base $[-1,1]$, we have that 
\[
\mathcal{E}(P) \not\subset  \{0\}\times Q.
\]
\item $Q$ is  neat.
\item  For every lattice smooth bundle with fiber $Q$ and an arbitrary base $B$, every point of $\mathcal{E}(B)$ can be lifted to a point in $\mathcal{E}(P)$.
\end{enumerate}	
\end{theorem}
	
\begin{proof}
We first prove (i)$\Rightarrow$(ii):
	For each $b\in \Z^m$ we have that 
		\[
		P=\conv((  \{1\} \times Q_b) \cup (  \{-1\}\times Q_{-b}))
		\]
		is a fiber product with fiber $Q$ and base $[-1,1]$. By hypothesis, there is $x\in \mathcal{E}(P)$ that is contained in $\{1\}\times Q_+ $, which implies $-x\in \{-1\}\times Q_- .$ Forgetting the first coordinate we obtain the desired point for the definition of neat.

For (ii)$\Rightarrow$(iii), let 
\[
B=\{x\in\R^{k}:Ux\le c_1\}, \qquad Q=\{y\in\R^{n}:Tx\le c_2\}.
\]
By Proposition~\ref{prop:bundle} (with Remark~\ref{rem:zero-in-B} into account), we assume
\[
P=\left\{(x,y)\in\R^{k+n}: \mymatrix{U&0\\S&T}\mymatrix{x\\y}\le\mymatrix{c_1\\c_2}\right\},
\]
for a certain matrix $S$.

		Now let $v\in \mathcal{E}(B)$. We have that 
		\[
		\pi^{-1}(v) = \{v\} \times \{y:Ty\le c_2- Sv\} = \{v\} \times Q_{-Sv},
		\]
		and $Q_{-Sv}$ is normally isomorphic to $Q$.
		If $v\in \mathcal{E}(B)$, the same holds for $-v$: $\pi^{-1}(-v)=\{-v\}\times Q_{Sv}$ and $Q_{Sv}$ is normally isomorphic to $Q$. As $Q$ is neat, there is a
		\[
		w \in \Z^n \cap Q_{-Sv} \cap (-Q_{Sv}),
		\]
		which gives
		\[
		(v,w) \in \mathcal E(P) \cap \pi^{-1}(v).
		\]
		
The implication (iii)$\Rightarrow$(i) is obvious.
\end{proof}

\begin{corollary}\label{cor:pmsym}
	For a lattice smooth polytope $Q$ the following properties are equivalent:
\begin{enumerate}
\item $Q$ is  neat and satisfies the weak (resp.~star) Ewald condition.
\item Every lattice smooth bundle $P$ with fiber $Q$ and base $[-1,1]$ satisfies the weak (resp.~star) Ewald condition.
\item Every lattice smooth bundle $P$ with fiber $Q$ and an arbitrary base $B$ satisfies the weak (resp.~star) Ewald condition whenever $B$ satisfies it.
\end{enumerate}	
\end{corollary}

\begin{proof}
We first look at the weak Ewald condition.
For (i)$\Rightarrow$(iii) consider a lattice basis contained in $\mathcal E(B)$ and lift it to  $\mathcal E(P)$. This gives an independent set that, together with any lattice basis contained in $\{0\} \times \mathcal E(Q)$, gives a lattice basis contained in $\mathcal E(P)$. (iii)$\Rightarrow$(ii) is obvious, and for (ii)$\Rightarrow$(i):
\begin{itemize}
\item The Cartesian product $[-1,1]\times Q$ shows that $Q$ satisfies the weak Ewald (by Proposition \ref{prop:cartesian}).
\item Theorem~\ref{thm:pmsym} shows that $Q$ is neat.
\end{itemize}

For the star Ewald condition, the proofs of (iii)$\Rightarrow$(ii)$\Rightarrow$(i) are the same, so let us prove (i)$\Rightarrow$(iii), which is essentially \cite[Proposition 5.3]{McDuff-probes}.

A face $F$ of $P$ is the product of a face $F_B$ of $B$ and a face $F_Q$ of $Q$. If $F_Q=Q$,
	\[\Star(F_B\times Q)=\Star(F_B)\times Q\]
	so applying the star condition to $F_B$ to get a point in $\mathcal{E}(B)$, and then using Theorem~\ref{thm:pmsym} to lift 
	this point to $\mathcal{E}(P)$, we obtain the desired point. If $$F_Q\subsetneq Q,$$ applying the star condition to it we obtain a point $p$ in $\mathcal{E}(Q)$. The inclusion of $p$ in $P$ gives a point which is contained exactly in the facets $B\times F_Q'$, where $F_Q'$ is a facet containing $p$, which is what we want.
\end{proof}

We finally look at how the neatness of a fiber bundle relates to that of the base and the fiber.

\begin{theorem}\label{pmsym2}
	Let $P$ be a lattice smooth bundle with base $B$ and fiber $Q$. 
	If $B$ and $Q$ are both neat, then $P$ is neat.
\end{theorem}

\begin{proof}
	Let $$P=\{(x,y)\in\R^{k+n}:Ax\le c\}$$ and $$P_b=\{(x,y)\in\R^{k+n}:Ax\le c+b\}.$$
	We may assume, as in the proof of Theorem~\ref{thm:pmsym} that
	\[A=\begin{pmatrix}
	U & 0 \\ S & T
	\end{pmatrix},
	\quad
	c=\begin{pmatrix}
		c_1 \\ c_2
	\end{pmatrix},
	\quad
	b=\begin{pmatrix}
	b_1 \\ b_2
	\end{pmatrix}\]
	where
	\[
	B=\{x\in\R^{k}:Ux\le c_1\}, \qquad Q=\{y\in\R^{n}:Tx\le c_2\}.
	\]
	
	The equations for $P_b$ (and the fact that it is normally isomorphic to $P$) imply that $P_b$ is
	a bundle with base \[B_{b_1}=\pi(P_b)=\{x\in\R^k:Ux\le c_1+b_1\}\] and fiber \[Q_{b_2}=\pi^{-1}(0)\cap P_b=\{y\in\R^n:Ty\le c_2+b_2\}\] (and that the new base and fiber are normally isomorphic to the original ones). 
	The same statements hold for $P_{-b}$, $B_{-b_1}$ and $Q_{-b_2}$.
	
	By the hypothesis applied to $B$, we obtain a point $v\in B_{b_1}$ with $-v\in B_{-b_1}$. Now, the fibers 
	\[
	\pi^{-1}(v)\cap P_b=\{y\in\R^n:Ty\le c_2+b_2-Sv\} =Q_{b_2-Sv}
	\] 
	and 
	\[
	\pi^{-1}(-v)\cap P_{-b}=\{y\in\R^n:Ty\le c_2-b_2+Sv\} =Q_{-b_2+Sv}
	\] 
	are also normally isomorphic to $Q$. Applying  the hypothesis to $Q$, we get $$w\in\pi^{-1}(v)\cap P_b$$ with $$-w\in\pi^{-1}(-v)\cap P_{-b},$$ so $w\in P_b$ with $-w\in P_{-b}$, and we are done.
\end{proof}

\subsection{Simplex-segment bundles}
\label{subsec:ssb}

We here study in detail the bundles with base a simplex and fiber a segment.

\begin{definition}[Simplex-segment bundle, $\SSB(n,k)$]
\label{def-ssb}
\label{ssbdef}
A bundle with base the monotone simplex and fiber the monotone segment  is called an \SSB\ (short for \emph{simplex-segment bundle}). 
More concretely, for $n,k\in\N$ with $n\ge 2$ and $0\le k\le n-1$ we denote $\SSB(n,k)$ the \SSB\  with the following inequality description:
	\begin{eqllave}
		x_i & \ge -1 \quad \forall i\in\{1,2,\ldots,n\}; \\
		x_1 & \le 1; \label{SSB} \\
		kx_1+x_2+\ldots+x_n & \le 1.
	\end{eqllave}
Its vertices are
\[(-1,-1,\ldots,-1), (-1,n-1+k,\ldots,-1), \ldots, (-1,\ldots,-1,n-1+k),\] \[(1,-1,\ldots,-1), (1,n-1-k,-1,\ldots,-1), \ldots, (1,-1,\ldots,-1,n-1-k).\]
\end{definition}

\begin{proposition}\label{prop-SSB}
	Every \SSB\ of dimension $n\ge 2$ is equivalent to $\SSB(n,k)$ for some $k$ with $0\le k\le n-1$. 
\end{proposition}

\begin{proof}
	We may assume, without loss of generality, that $\pi$ is the projection that sends the polytope to $x_1$. This implies that $x_1\ge -1$ and $x_1\le 1$ are facets.
	
	Moreover, $\pi^{-1}(0)$ is the fiber, which in this case is a monotone simplex:
	\begin{llave}
		x_i & \ge -1 \quad \forall i\in\{2,\ldots,n\}; \\
		x_2+\ldots+x_n & \le 1.
	\end{llave}
	By Proposition~\ref{prop:bundle}, our bundle has the form
	\begin{llave}
		-1\le x_1 & \le 1; \\
		a_ix_1+x_i & \ge -1 \quad \forall i\in\{2,\ldots,n\}; \\
		kx_1+x_2+\ldots+x_n & \le 1.
	\end{llave}
	where $a_i$ and $k$ are integers. By the coordinate change 
	$$x_i \mapsto x_i-a_ix_1, i\ge 2$$ 
	(which is unimodular), we obtain the bundle in the form \eqref{SSB}. For this to be actually a bundle, we need $|k|\le n-1$. If $k<0$, we can change $k$ to $-k$ by the cordinate change
	$$x_1 \mapsto -x_1,$$
	so we may assume that $0\le k\le n-1$.
\end{proof}

This gives $n$ different bundles. For $k=0$ we recover the Cartesian product. See Figure \ref{fig-SSB} for the three \SSB\ in dimension $3$.

\begin{figure}[h]
	\includegraphics[width=0.24\linewidth,trim=3cm 0 3cm 0,clip]{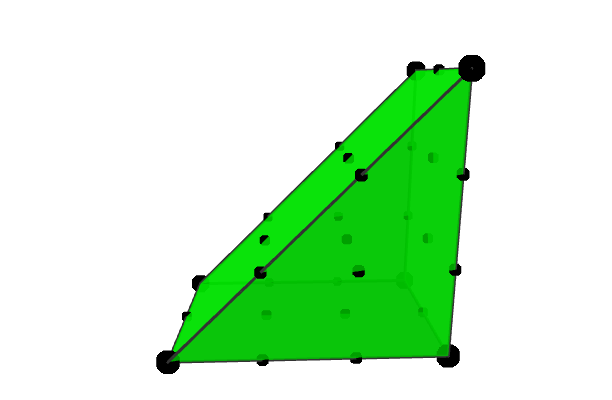}
	\includegraphics[width=0.24\linewidth,trim=3cm 0 3cm 0,clip]{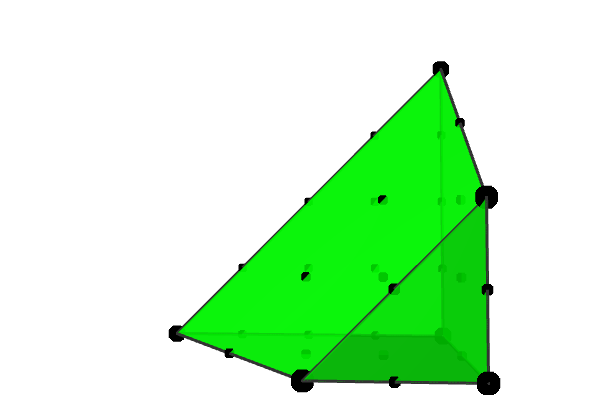}
	\includegraphics[width=0.24\linewidth,trim=3cm 0 3cm 0,clip]{ssb32}
	\caption{From left to right: $\SSB(3,0)$, $\SSB(3,1)$, $\SSB(3,2)$.}
	\label{fig-SSB}
\end{figure}

\begin{proposition} \label{prop-ut-ssb}
	$\SSB(n,k)$ is \UT-free if and only if $n=2$ or $k\le n-2$. It is deeply monotone if and only if $k\le 1$.
\end{proposition}

\begin{proof}
An \SSB\ is combinatorially the product of a simplex and a segment, with its two facets that are simplices having sizes $n+k$ and $n-k$. Hence, it has unimodular triangles if and only if ($n\ge 3$ and) $n-k=1$, that is, $k=n-1$.

A smooth $n$-simplex is deeply monotone if and only if its size is at least $n$. Since deep monotonicity is preserved under taking faces for $\SSB(n,k)$ to be deeply monotone we need $n-k\ge n-1$, that is, $k\le 1$. Sufficiency follows by induction on $n$, observing that the first displacements of the facets of $\SSB(n,k)$, for $k\le 1$, are the monotone simplex (for the facets that are simplices) and the rest give a monotone segment, if $n=2$, and $\SSB(n-1,k)$, if $n\ge 3$ and $k<n-1$.
\end{proof}

We can also check the Ewald properties for these polytopes:

\begin{lemma}\label{simplex-ewald}
	The monotone simplex satisfies the weak, strong and star Ewald conditions, and it is neat.
\end{lemma}

\begin{proof}
The Ewald conditions follow from it being deeply monotone (Theorem~\ref{thm:deep-ewald}).

By Theorem~\ref{thm:pmsym} that  it is neat follows from the fact that $\mathcal E(\SSB(n,k))$ has points with $x_1\ne 0$, such as 
\[
(1,\underbrace{-1,\dots,-1}_k,0,\dots,0).
\qedhere
\]
\end{proof}

\begin{proposition}\label{ssb-ewald}
	For every $n$ and $k$, $\SSB(n,k)$ satisfies the weak, strong and star Ewald conditions.
\end{proposition}
\begin{proof}
The weak and star Ewald conditions follow from Lemma~\ref{simplex-ewald} and Corollary~\ref {cor:pmsym}.

For the strong Ewald condition.
The points 
	\[(1,-1,\ldots,-1),(1,0,-1,\ldots,-1),\ldots,(1,-1,\ldots,-1,0)\]
	give the unimodular basis we need for the facet $x_1=1$, and their opposites give that for the facet $x_1=-1$. 
	Every other facet is equivalent to $x_2=-1$; its intersection with $x_1=0$ gives a facet of the unimodular simplex, and by Lemma \ref{simplex-ewald}, there is a unimodular basis of the hyperplane  $\{x_1=0\}$ contained in this facet. Adding the point $(1,-1,\ldots,-1)$, we have a basis of the whole space. 
\end{proof}

\section{The number of Ewald points} \label{s3}
\label{sec:number}

In this section we study the number of Ewald points in monotone polytopes.

\subsection{The maximum number of Ewald points}

It is easy to see that  the monotone cube $\mathrm C_n:= [-1,1]^n$ maximizes the number of Ewald points among all monotone polytopes:

\begin{proposition}
\label{prop:cube}
If $P$ is an $n$-dimensional monotone polytope then (modulo unimodular equivalence) 
	\[
	\mathcal{E}(P)\subset\mathcal{E}(\mathrm{C}_n).
	\]
	In particular,
	\[
	|\mathcal{E}(P)|\le 3^n,
	\] 
	and this bound is attained only if $P$ is the monotone cube.
\end{proposition}
\begin{proof}
	Without loss of generality, we may assume that we have a vertex at $$(-1,\ldots,-1)$$ with the facet normals in the coordinate directions. In that situation, all points of $P$ have all coordinates greater than $-1$, and
	\[
	\mathcal{E}(P)\subset P\cap -P\subset \mathrm{C}_n
	\]
	while $\mathcal{E}(\mathrm{C}_n)$ consists in all integer points of $\mathrm{C}_n$.
	
	This implies that $|\mathcal{E}(P)|\le 3^n$, with equality only if $P$ contains the monotone cube. But no monotone polytope can strictly contain the monotone cube by the following argument. Duality of polytopes reverses inclusions and the dual of the monotone cube is the \emph{standard cross-polytope}, which does not strictly contain any lattice polytope with the origin in its interior. (Its non-zero lattice points are $\{\pm e_i \,|\, i\in \{1,\dots, n\}\}$, and removing any of them results in a polytope not containing the origin in its interior).
\end{proof}

\subsection{Computation of the number of Ewald points}
\label{sec:computation}

We have computed the values of $|\mathcal{E}(P)|$ for all monotone polytopes of dimension at most seven. 
The full statistics of how many polytopes have each size for $n\le 5$ is in Table~\ref{table:sizes}. To save space, in dimension six and seven let us only mention that the minimum is $117$ and $243$, and the maximum, as predicted by Proposition~\ref{prop:cube}, is $3^6 = 729$ and $3^7=2187$, respectively.
The numbers in boldface in the first column correspond to the sizes that are triple of a size in the previous dimension, hence containing the products of those polytopes with segments. If the two counts coincide, i.e. all polytopes with the indicated size are products, the number in the second column is also in boldface.

\begin{table}
\begin{center}
\begin{tabular}[t]{|c|c|}\hline
	\multicolumn{2}{|c|}{$n=2$} \\ \hline
	$|\mathcal{E}(P)|$ & count \\ \hline
	7 & 4 \\
	\textbf{9} & \textbf{1}  \\ \hline
\multicolumn{2}{c}{}\\
\multicolumn{2}{c}{}\\
\hline 
	\multicolumn{2}{|c|}{$n=3$} \\ \hline
	$|\mathcal{E}(P)|$ & count \\ \hline
	13 & 2 \\
	17 & 9 \\
	19 & 2 \\
	\textbf{21} & \textbf{4} \\
	\textbf{27} & \textbf{1} \\ \hline
\end{tabular}
\qquad
\begin{tabular}[t]{|c|c|}\hline
	\multicolumn{2}{|c|}{$n=4$} \\ \hline
	$|\mathcal{E}(P)|$ & count \\ \hline
	27 & 2 \\
	31 & 4 \\
	33 & 14 \\
	35 & 3 \\
	37 & 6 \\
	\textbf{39} & 7 \\
	41 & 27 \\
	43 & 11 \\
	45 & 18 \\
	49 & 10 \\
	\textbf{51} & 15 \\
	\textbf{57} & \textbf{2} \\
	\textbf{63} & \textbf{4} \\
	\textbf{81} & \textbf{1} \\ 
	\hline
\end{tabular}
\qquad
\begin{tabular}[t]{|c|c|}\hline
	\multicolumn{2}{|c|}{$n=5$} \\ \hline
	$|\mathcal{E}(P)|$ & count \\ \hline
	59 & 3 \\
	61 & 2 \\
	63 & 3 \\
	65 & 1 \\
	67 & 7 \\
	69 & 4 \\
	71 & 1 \\
	73 & 15 \\
	75 & 13 \\
	77 & 6 \\
	79 & 65 \\
	\textbf{81} & 4 \\
	83 & 22 \\
	85 & 25 \\
	\hline
\end{tabular}
\hspace{-0.2cm}
\begin{tabular}[t]{|c|c|}\hline
	\multicolumn{2}{|c|}{$n=5$ (cont.)} \\ \hline
	$|\mathcal{E}(P)|$ & count \\ \hline
	87 & 41 \\
	89 & 8 \\
	91 & 30 \\
	\textbf{93} & 46 \\
	95 & 35 \\
	97 & 18 \\
	\textbf{99} & 87 \\ 
	101 & 19 \\
	103 & 79 \\
	\textbf{105} & \textbf{3} \\
	107 & 41 \\
	109 & 53 \\
	\textbf{111} & 33 \\
	113 & 13 \\
	\hline
\end{tabular}
\hspace{-0.2cm}
\begin{tabular}[t]{|c|c|}\hline
	\multicolumn{2}{|c|}{$n=5$ (cont.)} \\ \hline
	$|\mathcal{E}(P)|$ & count \\ \hline
	\textbf{117} & 18 \\
	119 & 36 \\
	121 & 36 \\
	\textbf{123} & \textbf{27} \\
	\textbf{129} & \textbf{11} \\
	133 & 8 \\
	\textbf{135} & \textbf{18} \\
	141 & 3 \\
	\textbf{147} & \textbf{10} \\
	\textbf{153} & \textbf{15} \\
	\textbf{171} & \textbf{2} \\
	\textbf{189} & \textbf{4} \\
	\textbf{243} & \textbf{1} \\ \hline
\end{tabular}
\end{center}
\caption{The number of Ewald points for all monotone polytopes of dimensions $2$, $3$, $4$ and $5$.}
\label{table:sizes}
\end{table}

In all dimensions products by segments tend to occupy the highest ranks, and among those that are not, the simplex is one of the highest. We conclude that the number of Ewald points does not appear to be related to  volume (the two cases of $13$ in dimension $3$ have greater volume than the cube). Instead, it is more related to a ``degree of symmetry'', understood, for example, as the volume of $P\cap -P$.

\begin{remark}
In dimension $2$ and $3$ the polytopes  with the same 
	value of $|\mathcal{E}(P)|$ have also the same $\mathcal{E}(P)$.
	For $n=2$, direct inspection of Figure \ref{fig-mon2} shows that the Ewald points of any monotone polygon other than the square  are the seven points in the monotone hexagon.	
	The square has, of course, nine Ewald points.
	
	In dimension $3$, the two polytopes with smallest $\mathcal{E}$ are those in Figure \ref{noutfree}, which only have $13$ Ewald points. The simplex and $\SSB(3,1)$, which is contained in it, have $19$ points. The Cartesian product of any $2$-monotone polytope with a segment has $21$, except the cube, that has $27$. 
The other nine monotone $3$-polytopes  have $17$ Ewald points.
\end{remark}

An interesting question is:

\begin{question}
What is the minimum of $|\mathcal{E}(P)|$ in each dimension, and which polytope attains it?
\end{question}

The minimum value could well be $1$ in some dimension (and hence in an infinite number of them by taking products). We will now make a construction that reaches the experimental minimum in all the dimensions where we know it.

\newcommand{\E}{\mathcal E}

Observe that if the fiber of a fiber product is a monotone simplex $\Delta_n$ then the description of 
Proposition~\ref{prop:bundle} can be simplified a bit: All but one of the functions $\phi_j$ in the statement can be assumed to be zero, so the fiber product is completely determined by the last one, that we can call just $\phi$. More explicitly, every monotone fiber product with fiber $\Delta_n$ and base $B\subset \R^k$ can be assumed to have description
\[
P:=\left\{ (x,y)\in \R^k\times \R^n\,|\, 
\begin{array}{rl}
	u_i\cdot x &\le  1, \ i =1,\dots, m \\ 
	-y_j & \le 1,  j =1,\dots, n \\
	\sum_j y_j &\le 1 + \phi(x) \\
\end{array}
\right\},
\]
where $\phi:\R^k\to \R$ is any linear function with $\phi(B \cap \Z^k) \subset \Z_{\ge -n}$. (The last condition is needed in order for $Q_x$ to be a positive integer dilation of $\Delta_n$).

In order for $P$ to have as few Ewald points as possible, the best choice seems to be to have $\phi$ as asymmetric as possible. 

\begin{definition}
	Let $B$ be a monotone polytope and $F$ a facet of it, with facet inequality $u\cdot x \le 1$.
	We call \emph{small fiber bundle with base $B$ and fiber $\Delta_n$ at the facet $F$}, the fiber bundle with the above description taking $\phi(x) = -n u\cdot x$.
\end{definition}

In what follows, when $F$ is a specified facet of a monotone polytope, with facet inequality $u\cdot x \le 1$, we denote
\begin{align*}
	\E_+(P,F)&:= \E(P) \cap F = \E(P) \cap \{u\cdot x = 1\},\\
	\E_0(P,F)&:=  \E(P) \cap \{u\cdot x = 0\},\\
	\E_-(P,F) &:=  \E(P) \cap \{u\cdot x = -1\}.
\end{align*}
Of course, $|\E_+(P,F)|=|\E_-(P,F)|$.

\begin{lemma}
	Let 
	$B$ be a monotone polytope and $F$ a facet of it. Let $P$ be the small fiber bundle with base $B$ and fiber $\Delta_n$ at the facet $F$, and let $F'$ be the facet with inequality $\sum_j y_j \le 1 + \phi(x)$ in $P$. Then:
	\begin{align*}
		|\E_+(P,F')|&= n  |\E_+(B,F)| + |\E_+(\Delta_{n})|\,|\E_0(B,F)|,\\
		|\E_0(P,F')|&=  2 |\E_+(B,F)| + |\E_0(\Delta_{n})|\,|\E_0(B,F)|.
	\end{align*}
\end{lemma}

Observe that in $\Delta_n$ we do not need to specify that facet for the notation $\E_+$ or $\E_0$, since all facets are equivalent. In fact,
\[
|\E_0(\Delta_{n})| = |\E(\Delta_{n-1})|, \qquad
|\E_+(\Delta_{n})| = (|\E(\Delta_{n})| - |\E(\Delta_{n-1})|)/2.
\]
The following table gives the values of this for small $n$:
\[
\begin{array}{r|cccccc}
	n& 1 & 2& 3& 4& 5 &6\\
	\hline
	|\E_0(\Delta_{n})| &1&3&7&19&51&141\\
	|\E_+(\Delta_{n})| &1&2&6&16&45&126\\
\end{array}
\]

\begin{proof}
	Let 
	\[
	B=\{x\in \R^n\,|\, u_i\cdot x \le 1, i =1,\dots, m\}
	\]
	be the inequality description of $B$, and $u\cdot x=1$ be the equation defining $F$. By definition,
	the inequality description of $P$ is
	\[
	P=\left\{ (x,y)\in \R^k\times \R^n\,|\, 
	\begin{array}{rl}
		u_i\cdot x &\le  1, \ i =1,\dots, m \\ 
		-y_j & \le 1,\  j =1,\dots, n \\
		\sum_j y_j &\le 1 - n u\cdot x. \\
	\end{array}
	\right\}.
	\]
	Consider the projection $\pi: P \to B$. We only need to look at how many points are there in the fiber of each $x\in \E(B)$. We only need to consider points with $u\cdot x \in \{0,+1\}$, since those with $u\cdot x =-1$ are opposite to the ones with $u\cdot x =1$, and those with $u\cdot x \not\in \{-1,0,+1\}$ do not contribute to $\E(P)$. 
	\begin{enumerate}
		\item Let $x\in \E_+(B,F) \subset F$, that is,  $u\cdot x =1$. Its fiber is then the unimodular simplex
		\[
		-\one + \delta_{n} = \{x\in \R^{n}\,|\, x_i\ge -1, \sum_i x_i \le 1-n\},
		\]
		which has $n+1$ lattice points (its vertices). Of these, $n$ of them (those with $\sum_i x_i = 1-n$) have their fibers in $\E_+(P,F')$ and the other one (the point $-\one$) in $\E_0(P,F')$. That is, each point of $\E_+(B,F)$ contributes $n$ to the count in $\E_+(P)$ and one to the count in $\E_0(P)$. The latter needs to be multiplied by two since the points in $\E_-(B,F)$ contribute to $\E_0(P)$ too.
		
		\item Let $x\in \E_0(B,F)$, that is $u\cdot x =0$. Its fiber is  the monotone simplex
		\[
		\Delta_{n} = \{x\in \R^{n}\,|\, x_i\ge -1, \sum_i x_i \le 1\},
		\]
		which contributes $|\E(\Delta_n)|$ points to $\E(P)$.
		Of these, $|\E_+(\Delta_n)|$ (resp. $|\E_0(\Delta_n)|$) lie in $|\E_+(P,F')|$ (resp. $|\E_0(P,F')|$).
\qedhere
	\end{enumerate}
\end{proof}

This has the following consequences:

\begin{lemma}
	Let $B$ be a monotone polytope and $F$ a facet of it with $|\E_+(B,F)|=|\E_0(B,F)|$. Then, 
	\begin{enumerate}
		\item The small fiber bundle $P$ with base $B$ and fiber $\Delta_3$ at the facet $F$ has $|\E(P)|= 9  |\E(B)|$ and still has a facet $F'$ with $|\E_+(P,F')|=|\E_0(P,F')|$.
		\item The small fiber bundle $P$ with base $B$ and fiber $\Delta_2$ at the facet $F$ has $|\E(P)|= \frac{13}{3}  |\E(B)|$.
		\item The second iteration of the small fiber bundle $P$ with base $B$ and fiber $\Delta_2$ at the facet $F$ has $|\E(P)|= \frac{59}{3}  |\E(B)|$.
	\end{enumerate}
\end{lemma}

\begin{corollary}\label{cor:minimum}
	For each  $n\in \Z_{\ge 1}$ let $\mathrm{E}_{\min}(n)$ denote the minimum number of Ewald points among monotone $n$-polytopes. 
Then:
	\begin{equation*}
\mathrm{E}_{\min}(n) \le 
\begin{cases}
3\cdot 9^k, & \text{if $n=3k+1$,}\\
\frac{59}9\cdot 9^k, & \text{if $n=3k+2$ and $k \ne 0$ (i.e., $n\ne 2$),}\\
13\cdot 9^k, & \text{if $n =3k+3$.}\\
\end{cases}
\end{equation*}
\end{corollary}

\begin{table}[h]
\small
	\begin{tabular}{|c|c|}\hline
		$n$ & upper bound \\ \hline
		3 & 13 \\
		4 & 27 \\
		5 & 59 \\
		6 & 117 \\
		7 & 243 \\
		8 & 531 \\
		9 & 1053 \\
		10 & 2187 \\
		11 & 4779 \\
		12 & 9477 \\ \hline
	\end{tabular}
\begin{tabular}{|c|c|}\hline
	$n$ & upper bound \\ \hline
	13 & 19683 \\
	14 & 43011 \\
	15 & 85293 \\
	16 & 177147 \\
	17 & 387099 \\
	18 & 767637 \\
	19 & 1594323 \\
	20 & 3483891 \\
	21 & 6908733 \\
	22 & 14348907 \\ \hline
\end{tabular}
\begin{tabular}{|c|c|}\hline
	$n$ & upper bound \\ \hline
	23 & 31355019 \\
	24 & 62178597 \\
	25 & 129140163 \\
	26 & 282195171 \\
	27 & 559607373 \\
	28 & 1162261467 \\
	29 & 2539756539 \\
	30 & 5036466357 \\
	31 & 10460353203 \\
	32 & 22857808851 \\ \hline
\end{tabular}
\caption{Values of the upper bound for $\mathrm{E}_{\min}(n)$.}
\end{table}

\subsection{Ewald points in the monotone simplex and the \SSB\ bundles.}
\label{subsec:SBB-number}

We can also compute this size for the simplex and for the \SSB. In what follows, we will use the notation \[[x^k]f(x)\] to mean the coefficient of $x^k$ in the polynomial $f(x)$.

\begin{proposition}
\label{prop:simplex}
	If $\Delta_n$ is the monotone $n$-simplex, $$|\mathcal{E}(\Delta_n)|=[x^{n+1}](1+x+x^2)^{n+1}.$$ For $n=1$ to $9$, this gives $3,7,19,51,141,393,1107,3139,8953.$\footnote{This sequence is registered as A002426 in the OEIS (Online Encyclopedia of Integer Sequences).
	In \url{https://oeis.org/A002426}, section ``Formula'', B. Cloitre and A. Mihailovs state without proof that
	\[\lim_{n\to\infty}\frac{\sqrt{n}}{3^n}|\mathcal{E}(\Delta_n)|=\sqrt{\frac{3}{4\pi}}.\]
	}
\end{proposition}
\begin{proof}
	Let $x$ be a lattice point in $\Delta_n$ such that $-x\in\Delta_n$. This implies that
	\begin{llave}
		-1 \le x_i & \le 1\quad\forall i\in[n]; \\
		-1 \le x_1+\ldots+x_n & \le 1,
	\end{llave}
	where $[n]$ stands for $\{1,\ldots,n\}$. Taking $x_{n+1}=-x_1-\ldots-x_n$, this becomes
	\begin{llave}
		-1 \le x_i & \le 1\quad\forall i\in[n+1]; \\
		x_1+\ldots+x_{n+1} & =0,
	\end{llave}
	and defining $y_i=x_i+1$,
	\begin{llave}
		0\le y_i & \le 2\quad\forall i\in[n+1]; \\
		y_1+\ldots+y_{n+1} & =n+1.
	\end{llave}
	The number of integer solutions to this system is equal to the aforementioned coefficient.
\end{proof}

\begin{theorem} \label{keyresult}
For $n\ge 2$, $0\le k\le n-1$, we have that
	$$
	|\mathcal{E}(\SSB(n,k))|=[x^n](1+x+x^2)^n+2[x^{n-k}](1+x+x^2)^n.
	$$
\end{theorem}
\begin{proof}
	Taking into account the equations of $\SSB(n,k)$, we have that all the symmetric points must have the coordinates between $-1$ and $1$. The symmetric points with $x_1=0$ are exactly those in the monotone $(n-1)$-simplex, that are counted by $$[x^n](1+x+x^2)^n.$$ For those with $x_1=1$, we have that
	\begin{llave}
		-1\le x_i & \le 1\quad\forall i\in[2,n]; \\
		-1-k\le x_2+\ldots+x_n & \le 1-k,
	\end{llave}
	where $[2,n]$ stands for $\{2,\ldots,n\}$. The lower bound of the first line and the upper bound in the second come from the fact that the point is in $\SSB(n,k)$, and the other bounds come from the symmetric point being also in $\SSB(n,k)$. We now take $x_{n+1}$ so that the sum is $-k$:
	\begin{llave}
		-1\le x_i & \le 1\quad\forall i\in[2,n+1]; \\
		x_2+\ldots+x_{n+1} & =-k,
	\end{llave}
	and making $y_i=x_i+1$,
	\begin{llave}
		0\le y_i & \le 2\quad\forall i\in[2,n+1]; \\
		y_2+\ldots+y_{n+1} & =n-k.
	\end{llave}
	The number of solutions is counted by $$[x^{n-k}](1+x+x^2)^n,$$ and the result follows.
\end{proof}

Theorem~\ref{keyresult} gives the following table of values for $|\mathcal{E}(\SSB(n,k))|$:

\medskip

\begin{center}
	\begin{tabular}{|c|c|c|c|c|c|c|c|c|c|} \hline
	\diagbox{$n$}{$k$}& 0 & 1 & 2 & 3 & 4 & 5 & 6 & 7 & 8 \\ \hline
	2 & 9 & 7 & & & & & & & \\ \hline
	3 & 21 & 19 & 13 & & & & & & \\ \hline
	4 & 57 & 51 & 39 & 27 & & & & & \\ \hline
	5 & 153 & 141 & 111 & 81 & 61 & & & & \\ \hline
	6 & 423 & 393 & 321 & 241 & 183 & 153 & & & \\ \hline
	7 & 1179 & 1107 & 925 & 715 & 547 & 449 & 407 & & \\ \hline
	8 & 3321 & 3139 & 2675 & 2115 & 1639 & 1331 & 1179 & 1123 & \\ \hline
	9 & 9417 & 8953 & 7747 & 6247 & 4903 & 3967 & 3451 & 3229 & 3157 \\ \hline
\end{tabular}
\end{center}

\medskip

We can notice several patterns in the table. We omit their (easy) proofs:
\begin{proposition}
\label{prop:ssb-number}
	For $n\ge 2$, the number of Ewald points of \SSB satisfies:
	\begin{enumerate}
		\item $|\mathcal{E}(\SSB(n,0))|=3|\mathcal{E}(\Delta_{n-1})|.$
		\item $|\mathcal{E}(\SSB(n,1))|=|\mathcal{E}(\Delta_n)|.$
		\item $|\mathcal{E}(\SSB(n,n-1))|=|\mathcal{E}(\Delta_{n-1})|+2n.$
		\item For a fixed $n$ and varying $k$, $|\mathcal{E}(\SSB(n,k))|$ decreases with $k$.
		\item In the same conditions, the volume of $\SSB(n,k)$ increases with $k$.
		\item For all $k$,%
		\footnote{The asymptotics for $|\mathcal{E}(\Delta_{n-1})|$  implies that, for any $k\in\N$,
	\[\sqrt{\frac{3}{4\pi}}\le \lim_{n\to\infty}\frac{\sqrt{n}}{3^n}|\mathcal{E}(\SSB(n,k))|\le 3\sqrt{\frac{3}{4\pi}}.\]
}
		\[
		1<\frac{|\mathcal{E}(\SSB(n,k))|}{|\mathcal{E}(\Delta_{n-1})|}\le 3.
		\]
	\end{enumerate}
\end{proposition}

\section{Nill's Conjecture: a proof for $n=2$ and partial results for higher $n$}\label{secgen}
\label{sec:Nill}

\begin{lemma}
\label{lemma:s=0}
Let $P\subset \R^2$ be a lattice polygon containing the origin. If $\mathcal{E}(P)=\{0\}$ then $P$ is unimodularly equivalent to
\[
T_a:=\conv\{(1,0), (0,1), (-a,-a)\},
\]
for some $a \in \Z_{> 0}$.
\end{lemma}

\begin{proof}
Suppose $\mathcal{E}(P)=\{0\}$. As an intermediate step, we are going to show that $P$ contains (after a suitable unimodular transformation) the triangle $T_1$.

Consider a unimodular triangulation of $P$ (which for a lattice polygon, always exists). The triangles containing the origin, extended as cones, form a smooth complete fan, with all  primitive generators contained in $P$.
Now, every smooth two-dimensional fan refines (after a unimodular transformation) either a ``Hirzebruch fan'' 
(the fan with with rays $(0,-1)$, $(0,1)$ $(-1,0)$, and $(1,k)$, for some $k\in \Z_{\geq 0}$) 
or the fan with generators $(1,0)$, $(0,1)$, $(-1,-1)$ (see \cite[Theorem 10.4.3]{CLS}); 
since all Hirzebruch fans contain two opposite generators, the condition $\mathcal{E}(P)=\{0\}$ implies that $P$ contains the three points $(1,0)$, $(0,1)$, $(-1,-1)$, as claimed.

Once  we know this, the fact that $(1,0),(0,1) \in P$ and $(1,1)=-(-1,-1)\not\in P$ implies that $P$ contains no lattice point in the interior of the positive quadrant. 
The same argument in the cones $\cone((-1,-1), (0,1))$ and $\cone((-1,-1), (0,1))$ implies that all lattice points of $P$ lie along one of the rays generated by $(1,0)$, $(0,1)$ and $(-1,-1)$.

Now, it is impossible for $P$ to contain points (other than the generators) along \emph{two} of these rays, because then it would contain also points in the interior of the cone. Thus, all lattice points (in particular all vertices) of $P$ apart of $\{(0,0), (1,0), (0,1),(-1,-1)\}$ lie along one of the rays of the fan. Without loss of generality we assume this ray to be the one generated by $(-1,-1)$.
\end{proof}

In the following statement we say that a vertex $v$ of a lattice polygon is \emph{quasi-smooth} if it lies at lattice distance one from the line going through $v^+$ and $v^-$, where $v^+$ and $v-$ are the lattice points previous and next to $v$ along $\partial P$.

\begin{lemma}
\label{lemma:quasi-smooth2}
Suppose $\{(-1,0),(1,0)\}\in P$ and that the (at most two) vertices of $P$ in the line $\{(x,y): y=0\}$ are quasi-smooth.
Then, $\mathcal{E}(P)$ contains some point of the form $(a,1)$. In particular, it contains a lattice basis.
\end{lemma}

\begin{proof}
We first prove the case when the line $\{(x,y): y=0\}$ contains no vertex of $P$. 
Let $l_i =\min\{x\in \R : (x,i) \in P\}$, for $i\in \{-1,0,1\}$. Observe that the $l_i$ may not be integer. Since $(l_0,0)$ is not a vertex, we have 
$l_{-1}+l_1= 2l_0 \le -2 $. Similarly, calling $r_i =\max\{x\in \R : (x,i) \in P\}$ we have $r_{-1}+r_1= 2r_0 \ge 2 $.
This implies that
\begin{align}
\lceil l_{-1}\rceil + \lceil l_1\rceil < 0  <
\lfloor r_{-1}\rfloor + \lfloor r_1\rfloor.
\label{eq:nonempty}
\end{align}

Now, $P$ must contain at least one lattice points of the form $(*,1)$ and one of the form $(*,-1)$ (as can be seen considering, for example, any unimodular triangulation of $P$ that 
uses the segment $\{(0,0), (1,0)\}$  as an edge.)
That is to say, $\lceil l_{1}\rceil \le \lfloor r_{1}\rfloor$ and $\lceil l_{-1}\rceil \le \lfloor r_{-1}\rfloor$.
Calling $I_1$ and $I_{-1}$ the intervals $[\,\lceil l_{1}\rceil , \lfloor r_{1}\rfloor\,]$ and $[\,\lceil l_{-1}\rceil , \lfloor r_{-1}\rfloor\,]$ respectively,  Equation \eqref{eq:nonempty} says that $0\in I_1+I_2$. Since $I_1$ and $I_2$ are integer intervals, this implies there is an $a\in I_1\cap \Z$ with $-a \in I_2$. Hence, $(a,1), (-a,-1) \in P$, as we wanted to show.

In the general case where the line $\{(x,y): y=0\}$ contains one or two vertices of $P$, the fact that these vertices are quasi-smooth implies that we can cut them off by respective lines at distance one
 and we have a smaller lattice polygon $P'$, which now does not have any vertex in the line $\{(x,y): y=0\}$. If $P'$ still contains the origin in its interior, then we apply the previous case to $P'$, so let us assume that it does not. 
 
 Then the origin must be in one of the two lines that we have used to cut $P$,
because the lines being at distance one implies that the triangles we have cut do not have interior lattice points, so if one of the triangles contains the origin then the origin is on the edge of the triangle opposite to the vertex of $P$ that we are cutting.
The fact that the origin is in one of those lines implies that first points along that line in opposite directions form a pair in $\mathcal{E}(P)$. That pair must be of the form $(a,1)$, $(-a,-1)$ for some $a$, because it forms a unimodular triangle with the origin and the vertex $(*,\pm1)$ of $P$.

 Hence, the  point $(a,1)$ along that line lies in $\mathcal{E}(P)$.
\end{proof}

These two lemmas easily imply the Generalized Ewald's conjecture in dimension two not only in the smooth case, but also in the quasi-smooth one:

\begin{corollary}[Generalized Ewald's conjecture in dimension 2] 
\label{coro:dim2}
If $P$ is a quasi-smooth lattice polygon with the origin in its interior then $\mathcal{E}(P)$ contains a lattice basis.
\end{corollary}

\begin{proof}
Since the lattice triangles $T_a$ of Lemma~\ref{lemma:s=0} are not quasi-smooth, we know that $\mathcal{E}(P)$ contains a non-zero point, and we can assume it to be $(1,0)$. Then, Lemma~\ref{lemma:quasi-smooth2} implies the statement.
\end{proof}

In the remaining of this section we prove two cases of the Generalized Ewald conjecture in higher dimension. In both of them we need to assume something about the position of the origin. 

First we extend Definition \ref{def:distance} to arbitrary faces.

\begin{definition}
\label{def:next}
Let $P$ be a lattice polytope, $F$ a face of it,  and $x_0\in P$. 
We call \emph{distance from $x_0$ to $F$} the maximum distance from $x_0$ to the facets containing $F$.
We say that $x_0$ is \emph{next to $F$} if it is in the interior of $P$ and at distance one from $F$.
Equivalently, if $x_0$ lies in the first displacement of $F$ (Definition~\ref{def:displacement}).
\end{definition}

\begin{proposition}
\label{prop:Nill-higherdim}
Let $P$ be a deeply smooth $d$-polytope with the origin in its interior, and suppose that the origin is next to a certain vertex $v$. Then, $\mathcal{E}(P)$ contains the lattice basis consisting of the primitive edge vectors of $P$ at $v$.
\end{proposition}

\begin{proof}
The proof is the same as in part (i) of Theorem~\ref{thm:deep-ewald}. By Corollary \ref{coro:deeply-faces} the first displacement of every edge at $v$ is a lattice segment, and since the origin is next to $v$ this lattice segment has the corresponding edge vector $u_i$ as an extreme point and the origin in its interior. Hence, $\{u_i,-u_i\} \in \mathcal{E}(P)$.
\end{proof}

In dimension three we can relax the hypotheses in two directions. We first need a lemma:

\begin{lemma}
\label{lem:quasisnooth}
Let $P$ be a  smooth $3$-polytope, let $F$ be a facet of it, and let $F_0$ be the first displacement of $F$. 
If $F_0$ is $2$-dimensional then it is a quasi-smooth polygon.
\end{lemma}

\begin{proof}
$F_0$ is a lattice polytope by Lemma~\ref{hyp-monotone}.
To prove that it is quasi-smooth (assuming it $2$-dimensional), let $v_0$ be a vertex of it. As seen in the proof of Lemma~\ref{hyp-monotone}(iii), the only case in which $v_0$ may not be smooth if it is the third vertex of a \UT-face of $P$ with an edge in $F$. Let $v_1$ and $v_2$ be the vertices of that edge, and let $v_1'$ and $v_2'$ be the lattice points next to $v_1$ and $v_2$ along the boundary of $F$. $v'_1$ and $v'_2$ must be different, since otherwise $P$ has two adjacent \UT-faces, and the only smooth $3$-polytope with two adjacent $\UT$-faces is the unimodular tetrahedron (so that $F_0$ would be a point). 

Since $F$ is smooth, the segment $v'_1v'_2$ is parallel to $v_1v_2$ and is at distance one from it. Consider now the lattice points  $v''_1$ and $v''_2$ next to $v_0$ along the boundary of $F_0$. These points must be 
\[
v''_i = v_0 +v'_i-v_i,
\]
for otherwise $v_0$ would also be the third vertex of a second $\UT$-triangle $v_0v_iv'_i$, which again would imply $P$ to be a unimodular tetrahedron. Also, we must have $v''_1\ne v''_2$ for otherwise $F_0$ would be a segment.
Hence, we have that $v_0$ is at distance one from the segment $v''_1v''_2$ with end-points at the lattice points next to $v_0$, which is the definition of being quasi-smooth.
\end{proof}

\begin{proposition}
\label{prop:Nill-dim3}
Let $P$ be a  smooth $3$-polytope with the origin in its interior, and suppose that the origin is next to a certain edge $uv$. Then, $\mathcal{E}(P)$ contains a lattice basis.
\end{proposition}

\begin{proof}
Let $F$ and $G$ be the two facets of $P$ containing $uv$, and let $F_0$ and $G_0$ be their first displacements, which contain the origin in their interior. By the previous lemma they are quasi-smooth, so we can apply Corollary~\ref{coro:dim2} to them. This tells us that 
$\mathcal{E}(F_0)$ and $\mathcal{E}(G_0)$ contain respective lattice bases of the two-dimensional lattices they span. Only one of the elements in these bases can coincide (namely the edge vector of $uv$) and talking the two vectors from one of the basis plus a non-coinciding one from the other gives us a basis of the three-dimensional lattice spanned by $P$.
\end{proof}

\section{Connection to symplectic toric geometry} \label{geometry}
\label{sec:toric}

The 
monotone polytopes we have studied in this paper are precisely the images under the momentum map of the so called monotone symplectic toric manifolds. Two of our main results (or rather, consequences of our main results)  are Theorem~\ref{thm:stem} and Theorem~\ref{thm:symplectic}, which exploit this
deep connection. In this section  we briefly recall this connection and how the problem of being displaceable for the fibers of the momentum map can be studied using the polytope.

\subsection{Monotone polytopes and symplectic toric manifolds} \label{del}

A \textit{symplectic toric manifold} is a quadruple
\[
(M,\omega,\T^n,\mu:M\to \R^n)
\]
where $(M,\omega)$ is a compact connected symplectic manifold of dimension $2n$, $\T^n$ is the standard $n$-dimensional torus which acts effectively and Hamiltonianly on $(M,\omega)$, and 
$$\mu:M\to \R^n$$ is the $\T^n$-action momentum map (which is uniquely defined up to translations and ${\rm GL}(n,\Z)$ transformations). 
By a theorem of Atiyah \cite{Atiyah} and Guillemin-Sternberg \cite{GuiSte} the image 
$\mu(M)\subset\R^n$ is a convex polytope, called the \emph{momentum polytope of $M$}, and given by the convex hull of the images under $\mu$ of the fixed points of the $\T^n$-action on $(M,\omega)$. 
Delzant~\cite{Delzant} classified  symplectic toric manifolds in terms of their momentum polytopes, by proving that that the application
\[
	(M,\omega,\T^n,\mu)\mapsto\mu(M)
\label{appmom}
\]
induces a bijection from the set of $2n$-dimensional symplectic toric manifolds, modulo isomorphism, to the set of 
smooth polytopes in $\R^n$ modulo unimodular equivalence (see for example \cite[section 4]{Pelsan-moduli} for precise definitions). For this reason smooth polytopes are often called \emph{Delzant polytopes} in the symplectic geometry literature.

 If in addition the first Chern class $c_1(M)$ of $M$ is equal (after normalization) to $[\omega]$, the symplectic toric manifold is called \textit{monotone}. 
 
 \emph{Monotone polytopes} are the polytopes associated to monotone symplectic toric manifolds via the bijection induced by (\ref{appmom}).  That is, a smooth polytope 
is \emph{monotone}  if it is 
(modulo the aforementioned normalization plus a lattice translation) the image under the momentum map of a monotone symplectic toric manifold.  See~\cite[Remark 3.2]{McDuff-probes} and \cite[p.~151, footnote]{McDuff-topology}.

\begin{example}
	The monotone $n$-simplex $\Delta_n$ corresponds to complex projective space
	$\C P^n$ endowed with the Fubini-Study form, and the monotone cube $[-1,1]^n$ corresponds to the product of $n$ copies of the complex projective line $\C P^1$.
	The other three monotone polygons (see Figure \ref{fig-mon2}) correspond to the blow-ups of $\C P^2$ at one, two or three points.
\end{example}

We refer to  Cannas da Silva~\cite{AC} and McDuff--Salamon~\cite{McduffSalamon} for an 
 introduction to Hamiltonian group actions and symplectic toric manifolds, and to McDuff's paper~\cite{McDuff-topology} for
 an in depth study of the properties of \emph{monotone} symplectic toric manifolds and \emph{monotone} polytopes.  
 In dimension $4$, Delzant's classification was generalized by the second author and V\~u Ng\d oc to \emph{semitoric symplectic manifolds} (or \emph{semitoric integrable systems}). The classification still involves a polytope, but besides the polytope additional invariants are needed. See~\cite{PeVN09, PeVN11, PeVN-BAMS11,PES}.

\subsection{Displaceable fibers in symplectic toric geometry and the Ewald conditions}

The top-dimensional fibers of the momentum map of a symplectic toric manifold 
(that is, the regular $\T^n$-orbits) are \textit{Lagrangian} submanifolds of $(M,\omega)$ in the sense that $\omega$ vanishes along them (see Figure \ref{fig-momentum} for examples). These orbits correspond to the preimages $\mu^{-1}(u),u\in\Int(P)$, where $\Int(P)$ is the interior of the polytope $P$, 
and are diffeomorphic to $(S^1)^n$.

\begin{example}
\label{exm:fibers}
Let us analyze the triangle, the momentum polygon of $\C P^2$, in more detail. Since coordinates in $\C P^2$ are defined only modulo a scalar factor, for every $[z_0:z_1:z_2]\in \C P^2$ we can assume without loss of generality that $|z_0|^2+|z_1|^2+|z_2|^2=1$. If, moreover, we consider barycentric coordinates in $\Delta_2$, then the moment map is simply
\[
\begin{array}{cccc}
\mu: &\C P^2& \to&  \Delta_2 \cr
&[z_0:z_1:z_2 ]&\mapsto& \left(|z_0|^2,|z_1|^2,|z_2|^2\right).
\end{array}
\]

Hence, the fiber of each $t=(t_0,t_1,t_2)\in \Delta_2$  is
\[
\mu^{-1}(t) = \left\{\left[\alpha_0\sqrt{t_0}:\alpha_1\sqrt{t_1}:\alpha_2\sqrt{t_2}\,\right] \,|\, \alpha_i \in S^1\right\},
\]
where $S^1=\{z\in \C\,|\, |z| =1\}$.

If $t$ lies in the interior of the triangle, one of the $\alpha_i$ can be taken equal to $1$, so the fiber is an $(S^1)^2$. If $t$ lies along an edge then one one of the $t_i$ is zero, so its $\alpha_i$ is irrelevant, and a second one can be assumed equal to $1$; the fiber is an $S^1$. Finally, if $t$ is a vertex then two of the $\alpha_i$ are irrelevant and the third one can be taken to be $1$, so the fiber is just a point. This is illustrated in the left part of Figure~\ref{fig-momentum}, and the right part shows the same for $\C P^1\times \C P^1$.
\end{example}

\begin{figure}[h]
	\begin{tikzpicture}[scale=0.65]
		\fill[color=orange] (-1.5,-1.5)--(-1.5,3)--(3,-1.5)--(-1.5,-1.5);
		\draw[->] (0,6) node {\includegraphics[height=2.75cm]{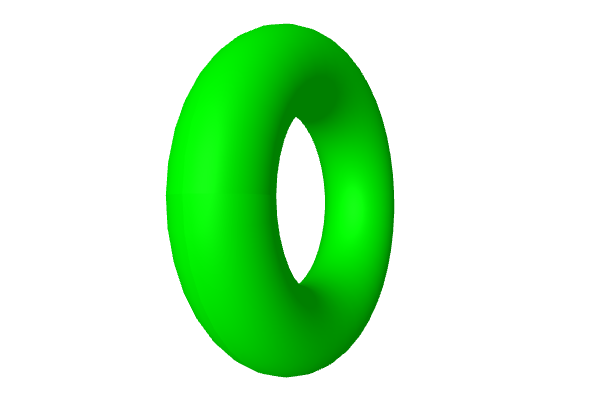}} -- (0,0);
		\draw[->] (3,4.2) -- (1.5,0);
		\draw[color=green] (3,6) ellipse (0.6cm and 1.5cm);
		\draw[->] (-3,4.5) -- (-1.5,3); 
		\fill[color=green] (-3.3,4.8) circle (0.05);
	\end{tikzpicture}
\hspace{1cm}
\begin{tikzpicture}[scale=0.65]
	\fill[color=orange] (-1.5,-1.5)--(-1.5,1.5)--(1.5,1.5)--(1.5,-1.5)--(-1.5,-1.5);
	\draw[->] (0,6) node {\includegraphics[height=2.75cm]{toro}} -- (0,0);
	\draw[->] (3,4.2) -- (1.5,0);
	\draw[color=green] (3,6) ellipse (0.6cm and 1.5cm);
	\draw[->] (-3,4.5) -- (-1.5,1.5); 
	\fill[color=green] (-3.3,4.8) circle (0.05);
\end{tikzpicture}
	\caption{
	The momentum polytopes of the complex projective space $\C P^2$ 
	and of $\C P^1\times \C P^1$ are a monotone triangle (left) and square (right). The toric fiber of a point in the interior is a 2-torus $(S^1)^2$, that of a point along an edge is a ``1-torus'' (that is, a circle) and that of a vertex is a point. See Example~\ref{exm:fibers}.}
		\label{fig-momentum}
\end{figure}

An important problem in symplectic topology, going back to Biran-Entov-Polterovich~\cite{BEP} and Cho~\cite{Cho} (see also Entov-Polterovich~\cite{EntPol}),  is deciding whether a fiber $$L_u:=\mu^{-1}(u),u\in\Int(P),$$ is \textit{displaceable} by a Hamiltonian isotopy, meaning that there exists a smooth family of functions $H_t:M\to\R, t\in[0,1]$, with associated flow $\phi_t$, and such that 
\[
\phi_1(L_u)\cap L_u=\varnothing.
\]

 The paper \cite{McDuff-probes} by  McDuff studies this question for
\textit{monotone} symplectic toric manifolds, exploiting the one-to-one correspondence between such manifolds and  monotone polytopes.
Entov and Polterovich \cite{EntPol} proved that in the monotone case the \emph{central fiber} (the fiber $L_{u_0}$ of the unique integral point in $u_0\in \mu(P)$, after normalization) is non-displaceable.%
\footnote{This was generalized by Fukaya-Oh-Ohta-Ono \cite{FOOO} as follows: for every toric symplectic manifold $M$ with momentum polytope $\mu(P)$ there is a point $u_0$, the so-called \textit{central point}, such that $L_{u_0}$ is not displaceable. The central point is defined, loosely speaking, as the unique point lexicographically maximizing its sequence of distances to facets, when the distances from a point to the facets are ordered from smallest to largest. In the monotone case the central point obtained by this procedure is the origin, since it is the only point at distance $\ge 1$ for every facet.
See also  \cite[Section 2.2]{McDuff-probes}.}
The main question addressed by McDuff is whether every other fiber is displaceable, in which case $L_{u_0}$ is called a \textit{stem}.

To attack this question McDuff introduces the notion of a point in a rational polytope being displaceable by a probe, that we now define.
In the following definition we say that a vector $\lambda\in \Z^n$ is \emph{integrally transverse} to a rational hyperplane $H$
if $\lambda$ can be completed to a unimodular basis by vectors parallel to $H$; equivalently, $\lambda = w-v$, where $w\in H$ and $v$ is at distance one from $H$, in the sense of Definition~\ref{def:distance}). 

\begin{definition}[{Displaceable by probes \cite[Definitions 2.3 and 2.5]{McDuff-probes}}]
\label{def:probes}
	Let $P$ be a rational polytope,  $F$ a facet of $P$, $w$ a point in the interior of $F$, and $\lambda\in\Z^n$ integrally transverse to $F$. The \textit{probe with direction $\lambda$ and initial point $w\in F$}
	 is the open line segment 
\[
p_{F,\lambda}(w):= w+ \R \lambda \cap \Int(P).
\]	

A point $u\in p_{F,\lambda}(w)$ is \emph{displaceable by the probe} if it is less than halfway along $p_{F,\lambda}(w)$; that is, if $2u-w\in \Int(P)$.
\end{definition}

The main results of McDuff relating the Ewald set to displaceabilty by Lagrangian isotopies are:

\begin{theorem}[McDuff]
\label{thm:McDuff}
\begin{enumerate}
\item Let $M$ be a toric symplectic manifold with momentum polytope $P$.
If a point $u\in \Int(P)$ is displaceable by a probe then its fiber $L_u\subset M$ is displaceable by a Hamiltonian isotopy~\cite[Lemma 2.4]{McDuff-probes}.
\item A monotone polytope $P$ has the star Ewald property if and only if every point of $\Int(P)\setminus \{0\}$ is displaceable by a probe~\cite[Theorem 1.2]{McDuff-probes}.
\end{enumerate}
\end{theorem}

Corollary \ref{cor:McDuff} follows from this.

We refer to ~\cite{AbMa, AC, LeRo, McduffSalamon}  for texts in symplectic geometry and its connection to mechanics.

\end{document}